\documentclass[11pt]{article}
\usepackage{graphicx,algorithmic,algorithm}%
\usepackage{enumerate}
\usepackage{tikz}
\usepackage{tikz-cd}
\usepackage{hyperref}
\hypersetup{
    pdftitle=   {Scattered classes of graphs},
   pdfauthor=  {O-joung Kwon and Sang-il Oum}
}
\usepackage{authblk}
\usepackage{amsthm,amssymb,amsmath}
\usepackage{mathabx}
\newtheorem{theorem}{Theorem}[section]
\newtheorem{lemma}[theorem]{Lemma}
\newtheorem{proposition}[theorem]{Proposition}

\newtheorem{conjecture}[theorem]{Conjecture}

\newtheorem*{thm:main1}{Theorem~\ref{thm:main1}}
\newtheorem*{thm:main2}{Theorem~\ref{thm:main2}}
\newtheorem*{thm:main3}{Theorem~\ref{thm:main3}}
\newtheorem*{thm:main4}{Theorem~\ref{thm:main4}}
\newtheorem*{thm:main5}{Theorem~\ref{thm:main5}}

\newcommand\abs[1]{\lvert #1\rvert}

\newcommand{\pivot}{\wedge}
\newcommand\mat{\boxminus}
\newcommand\rank{\operatorname{rank}}
\newcommand\tri{\boxslash}
\newcommand\antimat{\boxtimes}
\newcommand\tw{\operatorname{tw}}
\newcommand\pw{\operatorname{pw}}

\newcommand\vbrit{\beta^\vc}
\newcommand\ebrit{\beta^\ec}
\newcommand\mbrit{\beta^\matc}
\newcommand\rkbrit{\beta^\cutrk}

\def\K_#1{{K_{#1}}}
\def\S_#1{\overline{K_{#1}}}

\newcommand\vc{\kappa}
\newcommand\ec{\eta}
\newcommand\matc{\nu}
\newcommand\cutrk{\rho}
\newcommand\td{\operatorname{td}}

\newcommand\rd{\operatorname{rd}}
\newcommand\lrw{\operatorname{lrw}}
\newcommand\cover{\tau}
\newcommand\mw{\operatorname{mw}}

\newcommand{\cC}{\mathcal{C}}
\newcommand{\cF}{\mathcal{F}}

\begin{document}
\title{Scattered classes of graphs}

\author[1,2]{O-joung Kwon\thanks{Supported by IBS-R029-C1 and the National Research Foundation of Korea (NRF) grant funded by the Ministry of Education (No. NRF-2018R1D1A1B07050294), and the European Research Council (ERC) under the European Union's Horizon 2020 research and innovation programme (ERC consolidator grant DISTRUCT, agreement No. 648527). Part of the research took place while Kwon was at Logic and Semantics, Technische Universit\"{a}t Berlin, Berlin, Germany.}}

\author[2,3]{Sang-il Oum\thanks{Supported by IBS-R029-C1 and the National Research Foundation of Korea (NRF) grant funded by the Korea government (MSIT) (No. NRF-2017R1A2B4005020).}}

\affil[1]{Department of Mathematics, Incheon National University, Incheon,~South~Korea.}
\affil[2]{Discrete Mathematics Group, Institute for Basic Science (IBS), Daejeon,~South~Korea.}
\affil[3]{Department of Mathematical Sciences, KAIST,  Daejeon,~South~Korea.}

\date\today
\maketitle
\footnotetext{E-mail addresses: \texttt{ojoungkwon@gmail.com} (O. Kwon), \texttt{sangil@ibs.re.kr} (S. Oum)}

\begin{abstract}
  For a class $\mathcal C$ of graphs $G$
  equipped with functions $f_G$ defined on subsets of $E(G)$ or $V(G)$,
  we say that $\mathcal{C}$ is \emph{$k$-scattered} with respect to $f_G$ if 
  there exists a constant $\ell$ such that 
  for every graph $G\in \mathcal C$, 
  the domain of $f_G$ can be partitioned into subsets of size at most $k$
  so that the union of every collection of the subsets has $f_G$ value at most $\ell$.
  We present structural characterizations of graph classes that are $k$-scattered with respect to several graph connectivity functions.

  In particular, our theorem for cut-rank functions provides a rough structural characterization of graphs having no $mK_{1,n}$ vertex-minor, which allows us to prove that such graphs have bounded linear rank-width.
\end{abstract}

\section{Introduction}

	All graphs in this paper are undirected and simple.
	For a graph $G$, we write $V(G)$ and $E(G)$
        to denote the vertex set and edge set of $G$, respectively.

	In the theory of split decompositions, Cunningham~\cite{Cunningham1982} introduced the concept of a \emph{brittle graph}.
        A \emph{split} of a graph $G$ is a partition $(A,B)$ of the vertex set such that
        $\abs{A},\abs{B}\ge 2$ and
        no two vertices in $A$ have distinct nonempty sets of neighbors in $B$.
        Brittle graphs are connected graphs such that every vertex bipartition into two sets of size at least $2$ is a split.
	All brittle graphs are complete graphs or stars.
	Brittle graphs form basic classes of graphs in
        canonical split decompositions.

	Motivated by brittle graphs, 
	we introduce the general concept of a partition $(X_1, X_2, \ldots, X_m)$ of the vertex set or the edge set of a graph 
	such that each $X_i$ has at most $k$ elements, and for every $I\subseteq \{1,2, \ldots, m\}$, some connectivity measurement between $\bigcup_{i\in I}X_i$ and the rest is at most $\ell$, for given integers $k$ and $\ell$.  
	Brittle graphs then can be seen as graphs that admit a partition $(X_1, X_2, \ldots, X_m)$, where 
    $X_1, X_2, \ldots, X_m$ consist of distinct individual vertices, and 
    for every $I\subseteq \{1,2, \ldots, m\}$, the cut-rank function of $\bigcup_{i\in I}X_i$ is at most $1$. 
    This concept trades off between the allowed sizes of parts in a partition and the allowed values for a selected connectivity measurement. 

	We formally define this concept and provide examples. 
        Let $X$ be a finite set and $f:2^X\to\mathbb{Z}$. 
        The \emph{$f$-width} of a partition $(X_1,X_2,\ldots,X_m)$ of $X$, for some $m$,
        is \[
          \max\left\{ f\bigl(\bigcup_{i\in I} X_i\bigr): I\subseteq \{1,2,\ldots,m\}\right\}.
        \]
        The \emph{$k$-brittleness} of $f$
        is the minimum $f$-width of all partitions of $X$ into 
        parts of size at most $k$.

        We are mainly interested in the following four functions arising from graphs  naturally.

	\begin{itemize}
	\item For a subset $F$ of $E(G)$, let $\vc_G(F)$ be the number of vertices incident with both an edge in $F$ and an edge not in $F$.

        \item For a subset $S$ of $V(G)$, let $\ec_G(S)$ be the number of edges incident with both a vertex in $S$ and a vertex not in $S$.

        \item For a subset $S$ of $V(G)$, let $\matc_G(S)$ be the size of a maximum matching of a bipartite subgraph of $G$ obtained by taking edges joining $S$ and $V(G)\setminus S$.

        \item For a subset $S$ of $V(G)$, let $\cutrk_G(S)$ be the rank of the $S\times (V(G)\setminus S)$ $0$-$1$ matrix over the binary field whose $(a,b)$-entry for $a\in S$, $b\notin S$ is $1$ if $a$, $b$ are adjacent and $0$ otherwise. This function is called the \emph{cut-rank} function of $G$. (See Oum~\cite{Oum2004} for more properties of the cut-rank functions.)

        \end{itemize}

          The $k$-brittleness of $\vc_G$, $\ec_G$, $\matc_G$, $\cutrk_G$
          are called 
          the \emph{vertex $k$-brittleness} $\vbrit_k(G)$,
          the \emph{edge $k$-brittleness} $\ebrit_k(G)$,
          the \emph{matching $k$-brittleness} $\mbrit_k(G)$, 
          the \emph{rank $k$-brittleness} $\rkbrit_k(G)$
          of $G$, respectively.
          We say that a class $\cC$ of graphs is 
          \emph{vertex $k$-scattered} if
          the vertex $k$-brittleness of graphs in $\cC$ is bounded,
          \emph{edge $k$-scattered} if
          the edge $k$-brittleness of graphs in $\cC$ is bounded,
          \emph{matching $k$-scattered} if
          the matching $k$-brittleness of graphs in $\cC$ is bounded,
          and
          \emph{rank $k$-scattered} if
          the rank $k$-brittleness of graphs in $\cC$ is bounded.

	A class $\cC$ of graphs is called a \emph{subgraph ideal} if 
        it contains 
        every graph isomorphic to a subgraph of a graph in $\cC$.
        We characterize subgraph ideals which are vertex $k$-scattered,  edge $k$-scattered, or matching $k$-scattered.
        We remark that corresponding $k$-brittleness parameters do not increase by taking a subgraph.
        Our first theorem characterizes a vertex $k$-scattered subgraph ideal.
        For a graph $H$, we write $mH$ to denote the disjoint union of $m$ copies of $H$. A set $A$ of vertices is \emph{independent} if no two vertices in $A$ are adjacent. (Note that $\emptyset$ is independent.)
        For a graph $H$ and an independent set $A\subsetneq V(H)$,
        we write $mH/A$ to denote the graph obtained from $mH$
        by identifying all $m$ copies of each vertex in $A$ into one vertex.
        Note that the number of vertices of $mH/A$ is $m(\abs{V(H)}-\abs{A})+\abs{A}$
        and $1H/A=H$.
        See Figure~\ref{fig:mH} for an illustration.

        \begin{figure}
          \centering
   \begin{tikzpicture}
     \tikzstyle{v}=[circle,draw,fill=black!50,inner sep=0pt,minimum width=4pt]
     \draw (0,1) node [v] (a){};
     \draw (0,0) node [v] (d){};
     \foreach \i in {1,2,3,4}{
       \draw (2,-3+1+\i) node [v] (b\i) {};
       \draw (2,-3+0.5+\i) node [v] (c\i) {};
       \draw (a)--(b\i)--(c\i)--(d);
     }
     \draw (0,0.5) ellipse  (.5 and 1.5);
     \draw (-.7,0) node {$A$};
   \end{tikzpicture}          
          \caption{The graph $4P_4/A$ for a path $P_4=abcd$ with $A=\{a,d\}$.}
          \label{fig:mH}
        \end{figure}

	\begin{theorem}\label{thm:main1}
          Let $k$ be a positive integer. A subgraph ideal $\cC$ is vertex $k$-scattered if and only if
          \[
            \{1H/A, 2H/A, 3H/A, 4H/A, \ldots\}\not\subseteq \cC
          \]
          for every connected graph $H$ with exactly $k+1$ edges
          and each of its independent sets $A\subsetneq V(H)$
          such that $H- A$ is connected.
	\end{theorem}

	Our second theorem characterizes an edge $k$-scattered subgraph ideal.

	\begin{theorem}\label{thm:main2}
	Let $k$ be a positive integer. 	
	A subgraph ideal $\cC$ is edge $k$-scattered if and only if
        \[\{K_{1,1},K_{1,2},K_{1,3},\ldots\}\not\subseteq \cC\]
        and
        \[\{T, 2T, 3T, 4T, \ldots\}\not\subseteq \cC\]
        for every tree $T$ on $k+1$ vertices.
	\end{theorem}
	
        Our third theorem characterizes a matching $k$-scattered subgraph ideal. 

	\begin{theorem}\label{thm:main3}
	Let $k$ be a positive integer. 	
	A subgraph ideal $\cC$ is matching $k$-scattered if and only if
        \[
          \{T, 2T, 3T, \ldots\}\not\subseteq \cC
        \]
        for every tree $T$ on $k+1$ vertices.
	\end{theorem}
	
        Finally we characterize rank $k$-scattered graph classes. 
        As the cut-rank function may increase when we take a subgraph, 
        subgraph ideals are not suitable for the study of rank $k$-scattered graph classes.
	For instance, complete graphs are rank $1$-scattered and yet 
        an arbitrary graph is a subgraph of a complete graph.

        Instead of subgraphs, the containment relation called \emph{vertex-minors} 
        is more suitable for the study of rank $k$-scattered graph classes.
	A vertex-minor of a graph $G$ is an induced subgraph of a graph that can be obtained from $G$ by a sequence of \emph{local
  complementations}~\cite{Bouchet1987a,Bouchet1989a,Bouchet1990,Oum2004}, 
  where local complementation at a vertex $v$ is an operation to flip the adjacency relations between every pair of neighbors of $v$.
  The precise definition will be presented in Section~\ref{sec:prelim}.
  The cut-rank function is preserved when applying local complementations~\cite{Bouchet1989a,Oum2004}
  and therefore, the rank $k$-brittleness of a graph does not increase when taking vertex-minors.
	
	A class $\mathcal{C}$ of graphs is called a \emph{vertex-minor ideal} if 
it contains every graph isomorphic to a vertex-minor of a graph in $\mathcal C$.
	Our last theorem characterizes rank $k$-scattered vertex-minor ideals.

	\begin{theorem}\label{thm:main4}
	Let $k$ be a positive integer. 	
	A vertex-minor ideal $\cC$ is rank $k$-scattered if and only if
        \[
          \{ H, 2H, 3H, 4H, \ldots\}\not\subseteq \cC
        \]
        for every connected graph $H$ on $k+1$ vertices.
	\end{theorem}
	
        There are lots of interesting open problems on vertex-minors.
	In particular,
        the conjecture of Oum~\cite{Oum2006a}, if true,
        implies that for every circle graph $H$,
        every graph $G$ with sufficiently large rank-width
        has a vertex-minor isomorphic to $H$.
	This statement was known to be true
        when $G$ is a bipartite graph, a circle graph, or a line graph~\cite{Oum2004,Oum2006a}.
        Very recently, Geelen, Kwon, McCarty, and Wollan~\cite{GKMW2019}
        announced that they have a proof of this statement. Their proof uses our Theorem~\ref{thm:main4}
        as a starting point.
        
        Kant\'e and Kwon~\cite{KK2015} proposed
        the following analogous conjecture for linear rank-width.
        \begin{conjecture}[Kant\'e and Kwon~\cite{KK2015}]\label{conj:lrw}
          For every fixed forest $T$, there is an integer $f(T)$ such that
          every graph of linear rank-width at least $f(T)$ contains
          a vertex-minor isomorphic to $T$.
        \end{conjecture}

        By Ramsey's theorem,
        every sufficiently large connected graph contains one of $K_{1,n}$, $K_n$, or $P_n$ as an induced subgraph
	and if $n$ is huge, then
        each of these graphs contains a large star graph as a vertex-minor.
        Therefore for each fixed $n$, each component of a graph having no $K_{1,n}$ vertex-minor has bounded number of vertices 
	and thus it has bounded linear rank-width. Thus, Conjecture~\ref{conj:lrw}
        is true when $T$ is a star.

	We can strengthen this observation using Theorem~\ref{thm:main4}
        and verify Conjecture~\ref{conj:lrw} when $T$ is the disjoint union of stars.
	
	\begin{theorem}\label{thm:main5}
	For positive integers $m$ and $n$,  
	the class of graphs having no vertex-minor isomorphic to $mK_{1,n}$ has bounded linear rank-width.
        \end{theorem}
        Dahlberg, Helsen, and Wehner~\cite{DHW2019b} showed that it is NP-complete to decide whether a graph $G$ contains a vertex-minor isomorphic to a graph $H$, even if both $H$ and $G$ are restricted to circle graphs. However, we do not know the complexity of deciding whether a graph contains a vertex-minor isomorphic to a fixed graph $H$.
By Theorem~\ref{thm:main5}, we can recognize whether a graph contains
a vertex-minor isomorphic to the fixed disjoint
union of stars and complete graphs
in polynomial time. This works as follows. By Theorem~\ref{thm:main5},
if the input graph has large linear rank-width, then trivially
it has a vertex-minor isomorphic to $mK_{1,n}$ for some large $m$ and $n$
where $mK_{1,n}$ contains the disjoint union of stars and complete graphs as a vertex-minor.
Otherwise,
the input graph has bounded rank-width and so the theorem of Courcelle and Oum~\cite{CO2004} provides a polynomial-time algorithm.

	This paper is organized as follows. In Section~\ref{sec:prelim}, we
        present necessary definitions and notations.
        Section~\ref{sec:vertexbrittle} proves Theorem~\ref{thm:main1} for
        vertex $k$-scattered subgraph ideals, 
        Section~\ref{sec:edgebrittle} proves Theorem~\ref{thm:main2} for
        edge $k$-scattered subgraph ideals, 
        Section~\ref{sec:matchingbrittle} proves Theorem~\ref{thm:main3} for
        matching $k$-scattered subgraph ideals,
        and 
        Section~\ref{sec:rankbrittle} proves Theorem~\ref{thm:main4} for
        rank $k$-scattered vertex-minor ideals.
        Section~\ref{sec:parameter} compares our concepts with various other graph parameters.
        Section~\ref{sec:appl} discusses the application of Theorem~\ref{thm:main4}
        for linear rank-width, proving Theorem~\ref{thm:main5}.

	\section{Preliminaries}\label{sec:prelim}

	For a graph $G$ and a vertex set $S$ of $G$, we write $G[S]$ to denote the subgraph of $G$ induced by $S$. 	
	For $v\in V(G)$ and $S\subseteq V(G)$, $G- v$ is the graph obtained from $G$ by removing $v$ and all edges incident with $v$, and 
	$G- S$ is the graph obtained by removing all vertices in $S$. 
	For $F\subseteq E(G)$, $G-F$ is the subgraph of $G$ with the vertex set $V(G)$ and the edge set $E(G) \setminus F$.  
	For a vertex $v$ of a graph $G$, $N_G(v)$ is the set of \emph{neighbors} of $v$ in $G$, and the \emph{degree} of $v$ is the number of edges incident with $v$.
	For two disjoint vertex subsets $A$ and $B$ of $G$, 
	we write $G[A,B]$ to denote the bipartite subgraph on the bipartition $(A,B)$ consisting of all edges of $G$ having one end in $A$ and the other end in $B$.
        For two graphs $G$ and $H$, let $G\cup H$ be the graph $(V(G)\cup V(H), E(G)\cup E(H))$.

        A \emph{matching} of a graph is a set of edges of which no two edges share an end. For a matching $M$, we write $V(M)$ to denote the set of all vertices
        incident with an edge in $M$.
	A \emph{clique} in a graph is a set of pairwise adjacent vertices, 
	and an \emph{independent set} in a graph is a set of pairwise non-adjacent vertices. 
	
        The \emph{adjacency matrix} of a graph $G=(V,E)$, denoted by $A(G)$, 
        is a $V\times V$ $0$-$1$ matrix whose $(v,w)$ entry is $1$ if and only if $v$ and $w$ are adjacent.

	We write 
$P_n$ and $K_n$ to denote a path on $n$ vertices and a complete graph on $n$ vertices respectively.
	We write $K_{m,n}$ to denote a complete bipartite graph with bipartition $(A,B)$ where $\abs{A}=m$ and $\abs{B}=n$.
	For a graph $G$, we denote by $\overline{G}$ the \emph{complement} of $G$.

	We write $R(n;k)$ to denote the minimum number $N$ such
	that every coloring of the edges of $K_N$ into $k$ colors induces a monochromatic complete subgraph on $n$ vertices. Ramsey's theorem implies that $R(n;k)$ exists.

\paragraph{Vertex-minors}
	For a vertex $v$ in a graph $G$, performing a \emph{local complementation} at $v$ is to replace the subgraph of $G$ induced on $N_G(v)$ 
	by its complement graph. 	
	We write $G*v$ to denote the graph obtained from $G$ by applying
	a local complementation at $v$.
	Two graphs $G$ and $H$ are \emph{locally equivalent} if $G$ can be obtained from $H$ by a sequence of local complementations.
	A graph $H$ is a \emph{vertex-minor} of a graph $G$
	if $H$ is an induced subgraph of a graph locally equivalent to $G$.

\begin{figure}
   \tikzstyle{v}=[circle,draw,fill=black!50,inner sep=0pt,minimum width=4pt]
  \centering
     \begin{tikzpicture}[scale=0.9]
    \node[v](v1) at (0,.8){};
    \node[v](v2) at (0,1.6){};
    \node[v](v3) at (-.8,-.5){};
    \node[v](v4) at (.8,-.5){};
    \node[v](v5) at (-1.6,-1){};
    \node[v](v6) at (1.6,-1){};
      \draw (v1)--(v3)--(v4)--(v1);
      \draw (v1)--(v2);
      \draw (v3)--(v6);
      \draw (v2)--(v6);
      \draw (v3)--(v5);
      \draw (v4)--(v6);
      \draw (-0.3,.8) node{$v$};
      \draw (-0.6-0.3,-.3) node{$w$};
      \draw (0,-1.5) node{$G$};
    \end{tikzpicture}\qquad\qquad\qquad
         \begin{tikzpicture}[scale=0.9]
    \node[v](v1) at (0,.8){};
    \node[v](v2) at (0,1.6){};
    \node[v](v3) at (-.8,-.5){};
    \node[v](v4) at (.8,-.5){};
    \node[v](v5) at (-1.6,-1){};
    \node[v](v6) at (1.6,-1){};
      \draw (v1)--(v3)--(v4)--(v1);
      \draw (v1)--(v2);
      \draw (v3)--(v5);
      \draw (v1)--(v6);
      \draw (v3)--(v6);
      
      \draw (v2)--(v4);
      \draw (v5)--(v4);
      \draw (v2)--(v5);
      \draw (-0.3,.8) node{$w$};
      \draw (-0.6-0.3,-.3) node{$v$};
      \draw (0,-1.5) node{$G\wedge vw$};
    \end{tikzpicture}\caption{An example of pivoting.}
  \label{fig:pivotex}
\end{figure}

	For an edge $uv$ of a graph $G$, \emph{pivoting} the edge $uv$ in $G$
        is to take a series of three local complementations at $u$, $v$, and $u$.
        We write $G\pivot uv$ to denote the graph obtained by pivoting $uv$. 
        In other words, $G \pivot uv=G*u*v*u$.
	Note that $G\pivot uv$ is identical to the graph obtained from $G$ by flipping the adjacency relation between every pair of vertices $x$ and $y$ 
	where $x$ and $y$ are contained in distinct sets of $N_G(u)\setminus (N_G(v)\cup \{v\})$, $N_G(v)\setminus (N_G(u)\cup \{u\})$, and $N_G(u)\cap N_G(v)$, and finally swapping the labels of $u$ and $v$~\cite{Oum2004}.
	To \emph{flip} the adjacency relation between two vertices, we delete the edge if it exists and add it otherwise. 
	See Figure~\ref{fig:pivotex} for an example.
        For more details, see~\cite{Oum2004}.

\paragraph{Graph operations}
For two graphs $G$ and $H$ on disjoint vertex sets, each having $n$ vertices, we would like
to introduce operations to construct  graphs on $2n$ vertices by making the disjoint
union of them and adding some edges between two graphs. 
Roughly speaking, $G\mat H$ will add a perfect matching,
 $G\antimat H$ will add the complement of a perfect matching, 
 and $G\tri H$ will add a bipartite chain graph.
Formally, 
for two $n$-vertex graphs $G$ and $H$ with fixed ordering on the vertex sets $\{v_1,v_2,\ldots,v_n\}$ and $\{w_1, w_2, \ldots, w_n\}$ respectively, 
let 
$G\mat H$, $G\antimat H$, $G\tri H$
be graphs on the vertex set $V(G)\cup V(H)$ 
whose subgraph induced by $V(G)$ or $V(H)$
is $G$ or $H$, respectively
such that
 for all $i,j\in \{1,2,\ldots,n\}$, 
\begin{enumerate}[(i)]
\item 
$v_iw_j\in E(G\mat H)$ if and only if $i=j$,
\item 
$v_iw_j\in E(G\antimat H)$ if and only if $i\neq j$,
\item 
$v_iw_j\in E(G\tri H)$ if and only if $i\ge j$.
\end{enumerate}
See Figure~\ref{fig:construction} for illustrations of $\K_5\mat\S_5$, $\K_5\antimat\S_5$, and $\K_5\tri \S_5$.
In each of the constructed graphs, we say that $v_i$ is \emph{matched with} $w_j$ when $i=j$.

\begin{figure}
 \tikzstyle{v}=[circle, draw, solid, fill=black, inner sep=0pt, minimum width=3pt]
  \centering
  \newcommand\Sfive[1]{    \begin{tikzpicture}[scale=0.6]
      \foreach \x in {1,...,5} {
        \node [v]  (v\x) at(0,-\x){};
        \node [v]  (w\x) at (2,-\x){};
        \draw (-.5,-\x) node [left] {$v_\x$};
        \draw (w\x) node [right] {$w_\x$};
      }
      \draw (v1)--(v5);
      \draw(v1) [in=120,out=-120] to (v3);
      \draw(v2) [in=120,out=-120] to (v4);
      \draw(v3) [in=120,out=-120] to (v5);
      \draw(v1) [in=120,out=-120] to (v4);
      \draw(v2) [in=120,out=-120] to (v5);
      \draw(v1) [in=120,out=-120] to (v5);
      \foreach \x in {1,...,5} 
      \foreach \y in {1,...,5} {
        #1
      }
    \end{tikzpicture}
    }
    \Sfive{       \ifnum \x=\y  \draw (v\x)--(w\y);       \fi      }
    $\qquad$
    \Sfive{       \ifnum \x=\y  \else      \draw (v\x)--(w\y);       \fi      }
    $\qquad$
    \Sfive{       \ifnum \x>\y   \draw (v\x)--(w\y);       
    \else \ifnum\x=\y \draw (v\x)--(w\y);\fi\fi}
  \caption{$\K_5\mat \S_5$, $\K_5\antimat \S_5$, and $\K_5\tri \S_5$.}
  \label{fig:construction}
\end{figure}

	\section{Vertex $k$-scattered subgraph ideals}\label{sec:vertexbrittle}

	In this section, we characterize vertex $k$-scattered subgraph ideals.

	\begin{thm:main1}
	Let $k$ be a positive integer. A subgraph ideal $\cC$ is vertex $k$-scattered if and only if 
          \[
            \{1H/A, 2H/A, 3H/A, 4H/A, \ldots\}\not\subseteq \cC
          \]
          for every connected graph $H$ with exactly $k+1$ edges
          and each of its independent set $A\subsetneq V(H)$
          such that $H- A$ is connected.
	\end{thm:main1}
	
	For the forward part, we show the following.

        \begin{lemma}\label{lem:etabase1}
          Let $k$, $\ell$ be positive integers.
          Let $H$ be a connected graph with exactly $k+1$ edges.
          If $A$ is an independent set of $H$
          such that $H-A$ is connected,
          then the vertex $k$-brittleness of $(2\ell+1)H/A$ is at least $\ell+1$.
	\end{lemma}
        \begin{proof}
          Suppose not. Let $G=(2\ell+1)H/A$.
          Let $(X_1,X_2,\ldots,X_t)$ be a partition of $E(G)$
          such that its $\kappa_G$-width is at most $\ell$
          and $\abs{X_i}\le k$ for all $1\le i\le t$.

          For a component $C$ of $G-A$,
          let \[
            Y_C=\{ i\in\{1,\ldots,t\}:
            \text{some vertex in $V(C)$ is incident with an edge in $X_i$}\}.\]
          For each component $C$ of $G-A$,
          $\abs{Y_C}\ge 2$
          because $\abs{X_1},\abs{X_2},\ldots,\abs{X_t}\le k$
          and 
          vertices in $C$ are incident with more than $k$ edges in total.
          
          Let us pick a random subset $I$ of $\{1,2,\ldots,t\}$.
          For each component $C$ of $G-A$,
          the probability that $ Y_C\cap I\neq \emptyset$ and
          $Y_C\setminus I\neq\emptyset$ is $1-2^{1-\abs{Y_C}}\ge 1/2$.
          By the linearity of the expectation, there exists a subset $I'$ of $\{1,2,\ldots,t\}$
          such that
          at least $\ell+1$ components $C$ of $G-A$
          satisfy $Y_C\cap I'\neq \emptyset$ and
          $Y_C\setminus I'\neq \emptyset$.
          If $ Y_C\cap I'\neq\emptyset$ and $ Y_C\setminus I'\neq\emptyset$ for some component $C$,
          then $V(C)$ has a vertex incident with both an edge in $\bigcup_{i\in I'}X_i$ and an edge in $\bigcup_{i\notin I'}X_i$, because $C$ is connected.
          
          This means that $\kappa_G(\bigcup_{i\in I'} X_i) \ge \ell+1$, contradicting
          our assumption.
        \end{proof}

	For the converse direction of Theorem~\ref{thm:main1}, 
	we prove that 
	for positive integers $k$ and $n$, 
	every graph with sufficiently large vertex $k$-brittleness
        must contain a subgraph
        isomorphic to $nH/A$
        for some connected graph $H$ with $k+1$ edges
        and some independent set $A\subsetneq V(H)$
        such that $H-A$ is connected.
	We prove this statement by induction on $k$.
	The following lemma will be used in the induction step.
	\begin{lemma}\label{lemma:manyincidence}
          Let $H$ be a connected graph with exactly $k$ edges
          and let $A\subsetneq V(H)$ be an independent set
          such that $H-A$ is connected.
          Let $m$, $n$ be positive integers such that
          $m\ge 4(k+1)^2 n^2$.
          Let $G$ be a graph containing $mH/A$ as a subgraph.
          If           for each component $C$ of $(mH/A)-A$,
          $G$ has an edge not in $E(mH/A)$ but incident with vertices in $C$,
          then
          $G$ contains a subgraph isomorphic to $nH'/A'$
          for some connected graph $H'$ with $k+1$ edges
          and an independent set $A'\subsetneq V(H')$
          such that $H'-A'$ is connected.

	\end{lemma}
	\begin{proof}
          It is trivial if $n=1$. Thus we may assume that $n>1$.
          Let us choose a minimal subgraph $G'$ of $G$
          such that $V(G)=V(G')$,  $E(G')\cap E(mH/A)=\emptyset$
          and 
          for every component $C$ of $(mH/A)-A$,
          there is an edge in $G'$ incident with some vertex of $C$.
          Then $G'$ is a forest
          and $(V(G)\setminus V(mH/A)) \cup A$ is independent in $G'$ by the minimality.
          Moreover, between two components of $(mH/A)-A$, $G'$ has at most one edge
          and for each component $C$ of $(mH/A)-A$,
          the graph $G'[A\cup V(C)]$ has at most one edge.
          Moreover if $G'[A\cup V(C)]$ has an edge, then
          no other edges of $G'$ have exactly one end in $V(C)$.
          Let $m'$ be the number of components $C$ of $(mH/A)-A$
          such that $G'[A\cup V(C)]$ has no edge.

          Let $G''$ be the subgraph of $G'$
          obtained by deleting all edges $e$ having both ends in $V(C)\cup A$
          for some component $C$ of $(mH/A)-A$.
          As one edge of $G'$ is incident with at most two components,
          $G''$ has at least $m'/2$ edges
          and 
          $G'$ has at least $m'/2+(m-m')$ edges.

          If $m-m'>\binom{k+1}{2}(n-1)$,
          then by the pigeon-hole principle,
          there exists a pair of vertices $x$ and $y$ in $H$
          such that at least $n$ isomorphic copies of $H$
          in  $mH/A$
          has the copies $x'$, $y'$ of $x$ and $y$, respectively,
          such that $x'$, $y'$ are adjacent in $G'$.
          Then let $H'$ be the graph obtained from $H$ by adding $xy$.
          Then $G$ has $nH'/A$ as a subgraph.
          So, we may assume that
          $m-m'\le \binom{k+1}{2}(n-1)$.

          Note that vertices in $A$ are isolated in $G''$.
          If a vertex $v$ in $V(mH/A)$ has degree more than $1$ in $G''$,
          then
          no vertex in $G-V(mH/A)$ is adjacent to $v$ in $G''$
          because $G'$ is chosen to be minimal.
          Therefore all neighbors of $v$ in $G''$
          are in distinct components of $(mH/A)-A$.
          Notice that the same holds for a vertex $v$ outside of $mH/A$, because $(V(G)\setminus V(mH/A)) \cup A$ is independent in $G''$.

          If $G''$ has a vertex $v$ of degree more than $(k+1)(n-1)$, 
          then
          more than $(k+1)(n-1)$ components of $(mH/A)-A$ have vertices
          adjacent to $v$ in $G''$.
          By the pigeon-hole principle,
          there exists a vertex $w$ of $H-A$ such
          that in  at least $n$ components of $(mH/A)-A$,
          the copies of $w$ are adjacent to $v$ in $G''$.
          Let $H'$ be the graph obtained from $H$ by adding a new vertex $v$ of degree $1$ adjacent to $w$.
          Let $A'=A\cup \{v\}$.
          Then $G$ has $nH'/A'$ as a subgraph and
          both $H'$ and $H'-A'$ are connected.
          So we may assume that the maximum degree of $G''$ is at most $(k+1)(n-1)$.

          As $G''$ is a forest, $G''$ is bipartite. By K\"onig's theorem on the edge coloring, 
          $G''$ is $(k+1)(n-1)$-edge-colorable.
          So $G''$ has a matching $M$ with
          \[ \abs{M}\ge \frac{\abs{E(G'')}}{(k+1)(n-1)}\ge
            \frac{m'}{2(k+1)(n-1)}.\]

          Suppose that $m'>4(k+1)^2(n-1)^2$. 
          Let $C_1$, $C_2$, $\ldots$, $C_m$ be the components of $(mH/A)-A$.
          Let $I$ be a random subset of $\{1,2,\ldots,m\}$
          and $X=\bigcup_{i\in I} V(C_i)$.
          For each edge $e$ in $M$, 
          the probability that $e$ has exactly one end in $X$ 
          is $1/2$, no matter whether $e$ has one or two ends in $V(mH/A)$.
          Thus,
          there exist $I$ and $M'\subseteq M$ such that
          $\abs{M'}\ge \abs{M}/2>(k+1)(n-1)$
          and every edge of $M'$ has one end in $X$ and the other end not in $X$.
          By the pigeon-hole principle, there exists a vertex $u$ of $H-A$
          such that 
          at least $n$ edges $e$ of $M'$
          are incident with copies of $u$ in $mH/A$.
          Then let $H'$ be the graph obtained from $H$ by adding a new vertex $v$
          and an edge from $v$ to $u$
          and let $A'=A$. Then $G$ has $nH'/A'$ as a subgraph and
          both $H'$ and $H'-A'$ are connected.
          Therefore we may assume that $m'\le 4(k+1)^2(n-1)^2$.

          Then $m=m'+(m-m')\le 4(k+1)^2(n-1)^2+\binom{k+1}{2} (n-1)$.
          As $n-1<2n-1$ and $k/2< 4( k+1)$,
          we deduce that
          $ m< 4(k+1)^2(n-1)^2+ 4(k+1)^2 (2n-1)= 4(k+1)^2 n^2$. This contradicts our assumption on $m$.
        \end{proof}

        \begin{lemma}\label{lem:vbrittleconverse1}
          Every graph with vertex $1$-brittleness at least $256n^4$ contains
          $nP_3/A$ as a subgraph
          for some independent set $A\subsetneq V(P_3)$
          such that $P_3-A$ is connected.
        \end{lemma}
        \begin{proof}
          Let $G$ be a graph with vertex $1$-brittleness at least $256n^4$.
          We may assume that $G$ has no components with at most $2$ vertices.
          If $G$ has at least $n$ components, then
          each component has $P_3$ as a subgraph and therefore
          $nP_3/\emptyset$ is a subgraph of $G$.
          So we may assume that $G$ has less than $n$ components.

          Let $G'$ be the induced subgraph of $G$ obtained by deleting all degree-$1$ vertices.
          Then if a vertex of $G'$ has degree less than $2$,
          then it has its private neighbor in $V(G)\setminus V(G')$ of degree $1$ in $G$.

          If $G'$ has a vertex $v$ of degree at least $16n^2$,
          then $G'$ has $mP_2/\{v\}$ as a subgraph
          where $m$ is the degree of $v$ in $G'$.
          By Lemma~\ref{lemma:manyincidence},
          $G$ contains $nP_3/A$ for some independent set $A\subsetneq V(P_3)$ where $P_3-A$ is connected.
          So we may assume that every vertex of $G'$ has degree less than
          $16n^2$.

          If $G'$ has a matching $M$ of size at least $16n^2$,
          then $G'$ has $mP_2/\emptyset$ as a subgraph
          where $m=\abs{M}$.
          By Lemma~\ref{lemma:manyincidence},
          $G$ contains $nP_3/A$ for some independent set $A\subsetneq V(P_3)$ where $P_3-A$ is connected.
          So we may assume that every matching of $G'$ has less than  $16n^2$ edges.
          
          Then by the theorem of Vizing,
          $G'$ is $16n^2$-edge-colorable and therefore 
          $\abs{E(G')}\le 16n^2(16n^2-1)=256n^4-16n^2$.
          As $G'$ has at most $n-1$ components,
          $\abs{V(G')}\le \abs{E(G')}+n-1< 256n^4$.
          Then the vertex $1$-brittleness of $G$ is less than
          $256n^4$, which is a contradiction.
        \end{proof}

        For a set $A$ of vertices of a graph $G$, 
        a \emph{Tutte bridge} of $A$ in $G$ is
        either an edge joining two vertices in $A$
        or a connected subgraph of $G$ consisting of one component $C$ of $G-A$
        and all edges joining $C$ and $A$.
        Alternatively we may define a Tutte bridge as a connected subgraph of $G$
        induced by an equivalence class on $E(G)$
        where two edges $e$ and $f$ are equivalent if and only if
        there is a path starting with $e$ and ending with $f$ such that
        no internal vertex is in $A$.
        
        For a Tutte bridge $B$ of $A$ in $G$, \emph{deleting} $B$ from $G$
        is to remove all edges in $B$ and remove all vertices in $V(B)\setminus A$.
        Note that every component of $G$ is a Tutte bridge of $\emptyset$.
        The next lemma shows that the vertex $k$-brittleness does not decrease too much by deleting small Tutte bridges.
        \begin{lemma}\label{lem:removebridge}
          Let $G$ be a graph and  $A$ be a set of vertices of $G$.
          If $G'$ is the subgraph of $G$ obtained by deleting all Tutte bridges of $A$ having at most $k$ edges,
          then
          $\beta_k^\kappa(G')\ge \beta_k^\kappa (G)-\abs{A}$.
        \end{lemma}
        \begin{proof}
          Let $P'=(X_1,X_2,\ldots,X_t)$ be a partition of $E(G')$ whose $\kappa_{G'}$-width is
          equal to $\beta_k^\kappa (G')$.
          We extend $P'$ to a partition $P$ of $E(G)$
          by adding $E(B)$ as one part for each Tutte bridge $B$
          of $A$ in $G$ with at most $k$ edges.
          Then the $\kappa_G$-width of $P$ is at most $\beta_k^\kappa(G')+\abs{A}$
          and therefore $\beta_k^\kappa(G)\le \beta_k^\kappa(G')+\abs{A}$.
        \end{proof}
    
        To complete our proof, we will iteratively find an independent set $A_i$ 
        and two Tutte bridges of $A_i$ having at most $k$ edges for each $i$. 
        By combining two Tutte bridges, we will build a bigger connected subgraph, assuming that $A_i$ is nonempty. 
        Then we apply the sunflower lemma for the sets $A_1$, $A_2$, $\ldots$, which will allow us to find what we wanted.
        The next lemma allows us to find two Tutte bridges to be combined later.
        \begin{lemma}\label{lem:vbrittleconverse2}
          Let $m$, $n$, $k$ be positive integers.
          Let $H$ be a connected graph with $k$ edges and
          let $A\subsetneq V(H)$ be an independent set of $H$
          such that $H-A$ is  connected.
          Let $G$ be a graph having $mH/A$ as a subgraph
          such that
          no subgraph of $G$ is isomorphic to $nH'/A'$
          for some connected graph $H'$ with $k+1$ edges
          and an independent set $A'\subsetneq V(H')$ for
          which $H'-A'$ is connected.
          Let $X$ be a set of vertices of $G$.
          If $
          m> 4(k+1)^2n^2+\abs{X}
          $,
          then
          $G$ has two distinct Tutte bridges $B_1$, $B_2$ of $A$,
          satisfying the following.
          \begin{enumerate}[(i)]
          \item Each $B_i$ has exactly $k$ edges.
          \item $V(B_1)\cap A=V(B_2)\cap A=A$.
          \item neither $B_1-A$ nor $B_2-A$ has a vertex in $X$.
          \end{enumerate}
        \end{lemma}
        \begin{proof}
          By Lemma~\ref{lemma:manyincidence},
          less than $4(k+1)^2n^2$ components $C$ of $mH/A-A$
          are incident with an edge in $E(G)\setminus E(mH/A)$.
          Therefore there are at least $\abs{X}+2$ components  of $mH/A-A$
          that form Tutte bridges of $A$ in $G$ isomorphic to $H$.
          Among them, at least two, say $B_1$ and $B_2$, will not intersect~$X$.
          Since $H$, $H-A$ are connected and $A$ is independent in $H$,
          we deduce that $V(B_1)\cap A=V(B_2)\cap A=A$.
        \end{proof}
	We need the sunflower lemma. Let $\cF$ be a family of sets. 
	A subset $\{M_1, M_2, \ldots, M_p\}$ of $\cF$ is a \emph{sunflower} with \emph{core} $A$ (possibly an empty set) and $p$ \emph{petals} if
	for all distinct $i,j\in \{1,2, \ldots p\}$, $M_i\cap M_j=A$. 
	
	\begin{theorem}[Sunflower Lemma~{\cite[Erd\H{o}s and Rado]{ErdosR1960}}]\label{thm:sunflower}
	Let $k$ and $p$ be positive integers, and 
	$\cF$ be a family of sets each of cardinality $k$. 
	If $\abs{\cF}> k!(p-1)^k$, then $\cF$ contains a sunflower with $p$ petals.
        \end{theorem}
        Later we will apply Lemma~\ref{lem:vbrittleconverse2} iteratively
        and take $F_i:=B_1\cup B_2$ and $S_i:=A$ in the $i$-th round. Then we will apply the following lemma with $t:=2k$. Note that under this setting, $B_1\cup B_2$ is connected if $A$ is non-empty.
        \begin{lemma}\label{lem:sunflower}
          Let $m$, $n$, $k$, $t$ be positive integers.
          Let $G$ be a graph.
          For each $i\in \{1,2,\ldots,m\}$,
          let $F_i$ be a connected subgraph of $G$ with exactly $t$ edges
          having an independent set $S_i\subsetneq V(F_i)$
          such that $1\le \abs{S_i}\le k$ and 
          $F_i-X$ is connected for all $X\subsetneq S_i$.
          If
          $V(F_i)\cap V(F_j)\subseteq S_i\cap S_j$ and $S_i\neq S_j$
          for all $1\le i<j\le m$ and
          $m>k \cdot k! \binom{(t+1)t/2}{t}^k (n-1)^k $,
          then $G$ has a subgraph isomorphic to $nH/A$
          for some connected graph $H$ with exactly $t$ edges
          and an independent set $A\subsetneq V(H)$
          such that $H-A$ is connected.
        \end{lemma}
        \begin{proof}
          We may assume $n>1$ because otherwise
          we can take $H=F_1$ and $A=\emptyset$.
          Let $p=\binom{(t+1)t/2}{t}(n-1)+1\ge 2$.
          By the pigeonhole principle, more than $k! (p-1)^k$ of $S_1,S_2,\ldots,S_m$
          have the same cardinality.
          By Theorem~\ref{thm:sunflower},
          there exist $i_1<i_2<\cdots<i_p$
          such that $\{S_{i_1},S_{i_2},\ldots,S_{i_p}\}$ is
          a sunflower with $p$ petals and
          $\abs{S_{i_1}}=\abs{S_{i_2}}=\cdots=\abs{S_{i_p}}$.
          
          Let $A$ be the core, that is $A=\bigcap_{j=1}^p S_{i_j}$.
          Since $S_i\neq S_j$ for all $i\neq j$, 
          we have 
          $A\neq S_{i_j}$ for all $j\in\{1,2,\ldots,p\}$
          and therefore 
          $F_{i_j}-A$ is connected.

          Since $V(F_i)\cap V(F_j)\subseteq S_i\cap S_j$ for all $1\le i<j\le m$,
          we deduce that $F_{i_1}-A$, $F_{i_2}-A$, $\ldots$,
          $F_{i_p}-A$ are vertex-disjoint.
          There are at most $\binom{(t+1)t/2}{t}$ connected
          non-isomorphic graphs having exactly $t$ edges
          and so at least $n$ of $F_{i_1}$, $F_{i_2}$, $\ldots$,
          $F_{i_p}$ are pairwise isomorphic with isomorphisms fixing $A$,
          by the pigeonhole principle.
          This proves the lemma.
        \end{proof}
	\begin{proposition}\label{prop:vbrittleconverse}
          For positive integers $k$ and $n$, there exists an integer
          $\ell=\ell(k,n)$ such that
          every graph with vertex $k$-brittleness at least $\ell$ contains
          $nH/A$ as a subgraph
          for some connected graph $H$ with exactly $k+1$ edges
          and an independent set $A\subsetneq V(H)$
          such that $H-A$ is connected.

	\end{proposition}
	
	\begin{proof}
	We define that 
        \begin{align*}
          \ell(1, n)&:=256n^4,\\ 
          \intertext{	and for $k\ge 2$, }
           \ell(k, n)&:=\ell\left(k-1, 4(k+1)^2n^2 
                       +k^2 \cdot k! \binom{(2k+1)k}{2k}^k (n-1)^k     \right)
                       \\
            &\quad
                       +k^2 \cdot k! \binom{(2k+1)k}{2k}^k (n-1)^k  .
        \end{align*}

	We prove the statement by induction on $k$.
        If $k=1$, then it is true by Lemma~\ref{lem:vbrittleconverse1}.
	Now, we prove for $k\ge 2$.
	Suppose $G$ has vertex $k$-brittleness at least $\ell=\ell(k,n)$
        and no subgraph of $G$ is isomorphic to $nH'/A'$ for a connected graph $H'$ with
        $k+1$ edges having an independent set $A'\subsetneq V(H')$
        such that $H'-A'$ is connected.
        Let $m= 4(k+1)^2n^2+k^2 \cdot k! \binom{(2k+1)k}{2k}^k (n-1)^k 
        $.
        Let $G_1$ be the subgraph of $G$ obtained by deleting all components
        with at most $k$ edges. By Lemma~\ref{lem:removebridge}, $\beta_k^\kappa(G_1)=\beta_k^\kappa(G)$.
        Since $\ell(k,n)\ge\ell(k-1,m)$,
        by the induction hypothesis,
        $G_1$ has $mH_1/A_1$ as a subgraph
        for  some connected
        graph $H_1$ with $k$ edges having an independent set $A_1\subsetneq V(H_1)$
        such that $H_1-A_1$ is connected.
        Note that $\abs{A_1}\le k$.
        We may assume that $n\ge 2$, since $G_1$ has a component with more than $k$ edges.

        If $A_1=\emptyset$, then
        each component of $mH_1/A_1$ has a vertex incident (in $G_1$)
        with an edge not in $E(mH_1/A_1)$ because
        every component of $G_1$ has more than $k$ edges.
        By Lemma~\ref{lemma:manyincidence}.
        $G$ has a
        connected subgraph $H$ with an independent set $A$ having
        desired properties, contradicting our assumption.
        Therefore $A_1\neq\emptyset$.
        
        By Lemma~\ref{lem:vbrittleconverse2},
        $G_1$ has two Tutte bridges $B_{1,1}$ and $B_{1,2}$ of $A_1$, each having
        exactly $k$ edges such that
        $V(B_{1,1})\cap A_1=V(B_{1,2})\cap A_1=A_1$.
        Let $F_1=B_{1,1}\cup B_{1,2}$.
        Since $A_1\neq\emptyset$, $F_1$ is a connected graph.
        Then $F_1-X$ is connected  for all $X\subsetneq A_1$.

        For $i=\{2,\ldots,k\cdot k! \binom{(2k+1)k}{2k}^k (n-1)^k+1\}$,
        we define $G_i$ as the subgraph of $G_{i-1}$ obtained by
        deleting all Tutte bridges of $A_{i-1}$ having at most $k$ edges
        and then deleting all components having at most $k$ edges.
        By applying Lemma~\ref{lem:removebridge} twice,
        we deduce that 
        $\beta_k^\kappa(G_i)\ge \beta_k^\kappa (G_{i-1})-\abs{A_{i-1}}
        -\abs{\emptyset}
        \ge \beta_k^\kappa (G_{i-1})-k$.
        By induction,
        \[\beta_k^\kappa(G_i)\ge\beta_k^\kappa (G_1)
          - (i-1)k \ge  \ell(k-1,m),\]
        and by the induction hypothesis,
        $G_i$ has $mH_i/A_i$ as a subgraph for some connected
        graph $H_i$ with $k$ edges and an independent set $A_i\subsetneq V(H_i)$
        such that $H_i-A_i$ is connected.
        Note that $\abs{A_i}\le k$.
        If $A_i=\emptyset$, then each component of $mH_i/A_i$ has a vertex
        incident (in $G_i$) with an edge not in $E(mH_i/A_i)$
        because every component of $G_i$ has more than $k$ edges, contradicting the assumption by Lemma~\ref{lemma:manyincidence}.
        Thus $A_i\neq\emptyset$.
        Since
        \[m>4(k+1)^2n^2+ (i-1)k,\]
        by Lemma~\ref{lem:vbrittleconverse2},
        $G_i$ has two Tutte bridges $B_{i,1}$ and $B_{i,2}$ of $A_i$, each having
        exactly $k$ edges such that
        $V(B_{i,1})\cap A_i=V(B_{i,2})\cap A_i=A_i$
        and neither $B_{i,1}-A_i$ nor $B_{i,2}-A_i$ has a vertex
        in $A_1\cup A_2\cup \cdots \cup A_{i-1}$.
        Let $F_i=B_{i,1}\cup B_{i,2}$.
        As $A_i\neq\emptyset$, $F_i$ is a connected graph.
        Then $F_i-X$ is connected for all $X\subsetneq A_i$.

        We claim that for $i<j$, $V(F_i)\cap V(F_j)\subseteq A_i\cap A_j$.
        Suppose not. Let $x\in V(F_i)\cap V(F_j)$.
        When we construct $G_{i+1}$ from $G_i$, we remove all Tutte bridges of $A_i$ with at most $k$ edges, including all vertices of $F_i-A_i$.
        Since $F_j$ is a subgraph of $G_{j}$, 
        we deduce that $x\in A_i$.
        Because we choose $F_j$ so that $F_j-A_j$ has no vertex
        in $A_1\cup A_2\cup \cdots \cup A_{j-1}$ but $x\in A_i$,
        we conclude that $x\in  A_j$.
        This proves the claim.

        Suppose that $A_i=A_j$ for some $i<j$.
        By construction, $B_{j,1}-A_j$ has no vertex in $A_1\cup A_2\cup \cdots
        \cup A_{j-1}$.
        Note that $B_{j,1}$ is not a Tutte bridge of $A_i$ in $G_i$.
        So $G_i$ has an edge $e$ joining a vertex $v\in V(B_{j,1}-A_i)$
        to a vertex $w$ not in $V(B_{j,1})$.
        Note that $w\notin V(G_j)$ because $B_{j,1}$ is a Tutte bridge of $A_j$ in $G_j$.
        Let $p$ be the minimum integer such that
        $p\ge i$, $w\in V(G_p)$, and $w\notin V(G_{p+1})$.
        Since $B_{j,1}$ is a subgraph of $G_p$ and no vertex of $B_{j,1}-A_j$ is in $A_p$,
        all edges of $B_{j,1}$ together with $e$ are in the same Tutte bridge of $A_p$ in $G_p$, which has more than $k$ edges.
        Furthermore all edges of $B_{j,1}$ and $e$ are in the same component
        in the graph obtained from $G_p$ by deleting all Tutte bridges of $A_p$
        with at most $k$ edges.
        So $w$ is not deleted when constructing $G_{p+1}$, contradicting the assumption that $w\notin V(G_{p+1})$.
        Therefore $A_i\neq A_j$ for all $i<j$.

        By applying Lemma~\ref{lem:sunflower} to $F_i$ and $A_i$ for all $1\le i
        \le k\cdot k! \binom{(2k+1)k}{2k}^k (n-1)^k+1$,
        we deduce that $G$ has a subgraph isomorphic to $nH/A$
        for some connected graph $H$ with $2k$ edges having an independent set $A\subsetneq V(H)$ such that $H-A$ is connected.

        We claim that $H$ contains a connected subgraph $H'$ with
        exactly $k+1$ edges such that
        $H'-(A\cap V(H'))$ is connected.
        If $H-A$ has more than $k$ edges, then we can simply take $H'$ as a connected subgraph of $H-A$ with $k+1$ edges.
        If $H-A$ has at most $k$ edges, then
        let $H'$ be a connected
        subgraph of $H$ containing $H-A$ as a subgraph
        such that $H'$ has exactly $k+1$ edges.
        This proves the claim.
        However, this claim contradicts our assumption
        because $G$ contains $nH'/A'$ as a subgraph where $A'=A\cap V(H')$.
\end{proof}

	Lemma~\ref{lem:etabase1} and Proposition~\ref{prop:vbrittleconverse} imply Theorem~\ref{thm:main1}.
	\section{Edge $k$-scattered subgraph ideals}
		\label{sec:edgebrittle}
	In this section, we characterize edge $k$-scattered subgraph ideals.
	\begin{thm:main2}
	Let $k$ be a positive integer. 	
	A subgraph ideal $\cC$ is edge $k$-scattered if and only if
        \[\{K_{1,1},K_{1,2},K_{1,3},\ldots\}\not\subseteq \cC\]
        and 
        \[\{T, 2T, 3T, 4T, \ldots\}\not\subseteq \cC\]
        for every tree $T$ on $k+1$ vertices.
	\end{thm:main2}

	First we prove that for some connected graph $H$ on $k+1$ vertices, the disjoint union of sufficiently many copies of $H$ should have large edge $k$-brittleness.
	In fact, this is same for matching $k$-brittleness and rank $k$-brittleness, which we prove at the same time as follows.

	\begin{lemma}\label{lem:edgeforward1}
          Let $m$, $n$, $k$ be positive integers with $n>2m$ and
          $H$ be a connected graph on $k+1$ vertices.
          Then the following hold.
	\begin{enumerate}[(i)]
	\item $nH$ has edge $k$-brittleness at least $m+1$.	
	\item $nH$ has matching $k$-brittleness at least $m+1$.	
	\item $nH$ has rank $k$-brittleness at least $m+1$.	
	\end{enumerate}
	\end{lemma}
	\begin{proof}
          Let $G:=nH$.
          Let $(X_1, X_2, \ldots, X_t)$ be a partition of $V(G)$ such that
          $\abs{X_i}\le k$.
	Let $C_1, C_2, \ldots, C_n$ be the components of $G$.
        Note that each $C_i$ intersects at least two of $X_1$, $X_2$, $\ldots$, $X_t$.
        Let $I$ be a random subset of $\{1,2,\ldots,t\}$.
        For each $\ell$, the probability that $C_\ell$ contains
        both a vertex in $\bigcup_{i\in I}X_i$ and a vertex in $\bigcup_{j\in \{1,2,\ldots,t\}\setminus I} X_j$ is at least $1/2$.
        Thus, by the linearity of expectation,
        there exists $I\subseteq\{1,2,\ldots,t\}$ such that
        more than $m$ components of $G$ have 
        both a vertex in $\bigcup_{i\in I}X_i$ and a vertex in $\bigcup_{j\in \{1,2,\ldots,t\}\setminus I} X_j$.
        This implies that 
        $\eta_G(\bigcup _{i\in I} X_i)>m$,
        $\nu_G(\bigcup _{i\in I} X_i)>m$, and 
        $\rho_G(\bigcup _{i\in I} X_i)>m$.
	\end{proof}
	
	For edge $k$-brittleness, a large star is also an obstruction.
	\begin{lemma}\label{lem:edgeforward2}
          For positive integers $k$ and $m$,
          $K_{1,k+m}$ has edge $k$-brittleness at least $m+1$.
	\end{lemma}
	\begin{proof}
	Let $(X_1, X_2, \ldots, X_t)$ be a partition of $V(K_{1,k+m})$ such that 
        $\abs{X_i}\le k$.
	We may assume that  $X_1$ contains the center of $K_{1,k+m}$.
	Then $\ec_{K_{1,k+m}}(X_1)\ge (k+m)-(k-1)$.
 	\end{proof}
	
	Now, we show the backward direction of Theorem~\ref{thm:main2}.

	\begin{proposition}\label{prop:edgebrittleconverse}
	For all positive integers $k$ and $n$, there exists an integer $\ell=\ell(k,n)$ such that 	
	every graph with edge $k$-brittleness more than $\ell$ contains  a subgraph isomorphic to 
	either $K_{1,n}$ or $nT$ for some tree $T$ on $k+1$ vertices.
	\end{proposition}
	\begin{proof}
          Let $\ell(1,n)=n(n-1)$
          and $\ell(k,n)=\ell(k-1,4  k(n-1)^2+1)$ for $k\ge 2$.

          We proceed by induction on $k$.
          We may assume that every vertex has degree at most $n-1$.
          If $k=1$, then by the theorem of Vizing,
          $G$ has a matching of size at least $\abs{E(G)}/n$.
          Since
          the edge $1$-brittleness is less than or equal to $\abs{E(G)}$,
          $G$ has a matching of size more than $\ell(1,n)/n=n-1$,
          and so $G$ contains a subgraph isomorphic to $nK_2$.
          Thus, we may assume that $k>1$.

          We may assume that every component of $G$ has
          more than $k$ vertices, because otherwise removing them does not
          decrease the edge $k$-brittleness.
          By the induction hypothesis, $G$
          has a subgraph isomorphic to $mT$ for
          a tree $T$ on $k$ vertices
          where $m=4k(n-1)^2 +1$.
          Let $C_1$, $C_2$, $\ldots$, $C_m$ be the disjoint copies of $T$ in $G$.

          Let $G'$ be a minimal subgraph of $G$
          such that for all $1\le i\le m$, $G'$ has at least one edge
          joining $C_i$ with a vertex not in $C_i$.
          Since each edge of $G'$ is incident with at most two of $C_1$, $C_2$, $\ldots$, $C_m$,
          we have $\abs{E(G')}\ge \lceil m/2\rceil > 2k(n-1)^2$.
          Note that $G'$ is a forest.
          So by K\"onig's theorem on the edge coloring of bipartite graphs,
          $G'$ is $(n-1)$-edge-colorable and so
          it has a matching $M$ with $\abs{M}> 2k(n-1)$.
          Each edge of $M$ is incident with at least one copy of some vertex of $T$ in $mT$.

          Let $I$ be a random subset of $\{1,2,\ldots,m\}$.
          Let $X=\bigcup_{i\in I} V(C_i)$ and $Y=V(G)\setminus X$.
        The probability that an edge in $M$ has one end in $X$ and the other end in $Y$ is $1/2$ and therefore
        there exist $I$ and $M'\subseteq M$ such that
        $\abs{M'}\ge \abs{M}/2>k(n-1)$ and
        each edge of $M'$ has one end in $X$ and the other end in $Y$.

          Now $M'$ has a subset $M''$ with $\abs{M''}>n-1$
          such that
          there exists a vertex $w$ of $T$ with the property that 
          for every edge of $M''$, its end in $X$
          is a copy of $w$ in $mT$.
          Let $T'$ be the tree obtained from $T$ by adding a new vertex adjacent to
          $w$ only.
          Then $G$ has $nT'$ as a subgraph.
	\end{proof}
	
	Proposition~\ref{prop:edgebrittleconverse}
        and 
	Lemmas~\ref{lem:edgeforward1} and \ref{lem:edgeforward2}
        imply Theorem~\ref{thm:main2}.

	\section{Matching $k$-scattered subgraph ideals}
	\label{sec:matchingbrittle}
	In this section, we characterize matching $k$-scattered subgraph ideals.
	We already proved in Lemma~\ref{lem:edgeforward1} that 
	for a connected graph $H$ on $k+1$ vertices, the disjoint union of sufficiently many copies of $H$ has large matching $k$-brittleness.
	Such obstructions exactly characterize matching $k$-scattered subgraph ideals.
	
	\begin{thm:main3}
	Let $k$ be a positive integer. 	
	A subgraph ideal $\cC$ is matching $k$-scattered if and only if
        \[
          \{T, 2T, 3T, \ldots\}\not\subseteq \cC
        \]
        for every tree $T$ on $k+1$ vertices.
	\end{thm:main3}

        First let us prove that deleting a vertex does not decrease
        the matching $k$-brittleness a lot.
        \begin{lemma}\label{lem:delmat}
          Let $k$ be a positive integer. For each vertex $v$ of a graph $G$, 
          \[\beta_k^\nu (G)\le \beta_k^\nu(G-v)+1.\]
        \end{lemma}
        \begin{proof}
          Let $P'=(X_1,X_2,\ldots,X_t)$ be a partition of $V(G-v)$
          such that $\abs{X_i}\le k$
          and the $\nu_{G-v}$-width of $P'$ is minimum, that is  $\beta_k^\nu(G-v)$.
          Let $P=(X_1,X_2,\ldots,X_t,\{v\})$.
          Then the $\nu_G$-width of $P$ is at most  $\beta_k^\nu(G-v)+1$.
        \end{proof}
        
        The following proposition with Lemma~\ref{lem:edgeforward1}
        proves Theorem~\ref{thm:main3}.

	\begin{proposition}\label{prop:matchingbrittleconverse}
	For all positive integers $k$ and $n$, there exists $\ell=\ell(k,n)$ such that 	
	every graph with matching $k$-brittleness more than $\ell$ contains a subgraph isomorphic to 
	$nT$ for some tree $T$ on $k+1$ vertices.
	\end{proposition}
	\begin{proof}
          Let $\ell(k,n)= (k+1)^k (n-1)$. Let $G$ be a graph with matching $k$-brittleness more than $\ell(k,n)$.
          Let $G_0=G$ and $S_0=\emptyset$.
          We claim that there exist disjoint subsets $S_1$, $S_2$, $\ldots$,
          $S_{(k+1)^{k-1}(n-1)}$,           $S_{(k+1)^{k-1}(n-1)+1}$
          such that each $S_i$ induces a connected subgraph of $G$
          with $k+1$ vertices.
          For $i=1,2,\ldots, (k+1)^{k-1}(n-1)+1$,
          let $G_i$ be the induced subgraph of $G_{i-1}-S_{i-1}$
          obtained by deleting all
          components with at most $k$ vertices.
          Notice that by Lemma~\ref{lem:delmat},
          $\beta_k^\nu (G_i)\ge \beta_k^\nu (G_{i-1})-\abs{S_{i-1}}
          = \beta_k^\nu (G_{i-1})-(k+1)$.
          By induction, we deduce that
          $\beta_k^\nu (G_i)\ge \beta_k^\nu (G)- (k+1) (i-1)>0$.
          Thus $G_i$ contains a component with more than $k$ vertices
          and therefore it has a vertex set $S_i$ of size $k+1$ inducing
          a connected subgraph.
          This proves the claim.

          Let $T_i$ be a spanning tree of $G[S_i]$ for each $i$.
          Since the number of labeled
          trees on $k+1$ vertices is $(k+1)^{k-1}$,
          there exist more than $n-1$ of these spanning trees
          that are pairwise isomorphic.
		\end{proof}

	\section{Rank $k$-scattered vertex-minor ideals}
\label{sec:rankbrittle}

	We characterize rank $k$-scattered vertex-minor ideals. As we mentioned, the rank $k$-brittleness of a graph may increase when taking a subgraph.
        Instead we use vertex-minors because of the following lemma.
        \begin{lemma}[See Oum~{\cite[Proposition 2.6]{Oum2004}}]\label{lem:loc}
          If $G$ is locally equivalent to $G'$, then for every subset $X$ of vertices of $G$,
          $\rho_G(X)=\rho_{G'}(X)$.
        \end{lemma}

        Here is our main theorem for rank $k$-scattered vertex-minor ideals.
	\begin{thm:main4}
	Let $k$ be a positive integer. 	
	A vertex-minor ideal $\cC$ is rank $k$-scattered if and only if
        for every connected graph $H$ on $k+1$ vertices,
        \[
          \{ H, 2H, 3H, 4H, \ldots\}\not\subseteq \cC.
        \]
	\end{thm:main4}
        
        First, it is easy to observe the following.
        \begin{proposition}\label{prop:vmbrittle}
          If $H$ is a vertex-minor of $G$, then
          \[\beta_k^\rho (G)\le \beta_k^\rho(H)+\abs{V(G)}-\abs{V(H)}.\]
        \end{proposition}
	\begin{proof}
	Let $G'$ be a graph locally equivalent to $G$ such that $H$ is an induced subgraph of $G'$.
	Note that applying local complementation does not change the rank $k$-brittleness of a graph by Lemma~\ref{lem:loc}.
	Therefore, we have $\rkbrit_k(G')=\rkbrit_k(G)$.
        It is easy to observe that
        removing a vertex may decrease the rank $k$-brittleness by at most $1$
        by a proof analogous to the proof of Lemma~\ref{lem:delmat}.
	Therefore, 
	$\rkbrit_k(H)\ge \rkbrit_k(G')-(\abs{(V(G')}-\abs{V(H)})= \rkbrit_k(G)-(\abs{(V(G)}-\abs{V(H)})$, as required.
	\end{proof}

	Lemma~\ref{lem:edgeforward1} states that
	for a connected graph $H$ on $k+1$ vertices, the disjoint union of sufficiently many copies of $H$ has large rank $k$-brittleness.
	It means that if $\{ H, 2H, 3H, 4H, \ldots\}\subseteq \cC$ for some connected graph $H$ on $k+1$ vertices, 
	then $\cC$ is not rank $k$-scattered.
 	Now we focus on the other direction of Theorem~\ref{thm:main4}.
	We need the following Ramsey-type theorem for bipartite graphs without twins.
\begin{theorem}[Ding, Oporowski, Oxley,
  Vertigan~\cite{GuoliBJD1996}]\label{thm:largebipartite}
     For every positive integer $n$,
     there exists an integer $f(n)$ 
     such that for every bipartite graph $G$ with a bipartition $(S,T)$, 
     if no two vertices in $S$ have the same set of neighbors
     and      $\abs{S}\ge f(n)$,  
     then $S$ and $T$ have $n$-element subsets $S'$ and $T'$, respectively, such that 
     $G[S',T']$ is isomorphic to $\S_{n}\mat\S_{n}$, $\S_{n}\tri \S_{n}$, or
     $\S_{n}\antimat\S_{n}$. 
\end{theorem}

	In several places of the proof, when we obtain $H_1\mat H_2$ or $H_1\antimat H_2$
	where $H_1, H_2\in \{\S_n, \K_n\}$, 
	we want to make each part an independent set.
	The following lemma describes how to reduce each of them
        to $\S_{n'}\mat \S_{n'}$ for some $n'$.

	\begin{lemma}\label{lem:tomatching}
Let $n$ be an integer.
\begin{enumerate}[(1)]
	\item If $n\ge 2$, then $\K_n\mat\S_n$ has a vertex-minor isomorphic to $\S_{n-1}\mat\S_{n-1}$.
	\item If $n\ge 3$, then $\K_n\mat\K_n$ has a vertex-minor isomorphic to $\S_{n-2}\mat\S_{n-2}$.
	\item If $n\ge 3$, then $\S_n\antimat\S_n$ has a vertex-minor isomorphic to $\S_{n-2}\mat\S_{n-2}$.
	\item If $n\ge 3$, then $\K_n\antimat\S_n$ has a vertex-minor isomorphic to $\S_{n-2}\mat\S_{n-2}$.
	\item If $n\ge 2$, then $\K_n\antimat\K_n$ has a vertex-minor isomorphic to $\S_{n-1}\mat\S_{n-1}$.
\end{enumerate}
\end{lemma}
\begin{proof}
(1) Let $V(\K_n)=\{v_i:1\le i\le n\}$ and $V(\S_n)=\{w_i:1\le i\le n\}$.
	The graph $(\K_n\mat\S_n-w_1)*v_1-v_1$ is isomorphic to $\S_{n-1}\mat\S_{n-1}$.
     
    (2) Let $\{v_i:1\le i\le n\}$ and $\{w_i:1\le i\le n\}$ be the vertex sets of two copies of $\K_n$. 
    The graph $((\K_n\mat\K_n-\{v_1, w_2\})*v_2*w_1)-\{v_2, w_1\}$ is isomorphic to $\S_{n-2}\mat\S_{n-2}$.
    
    (3) Let $\{v_i:1\le i\le n\}$ and $\{w_i:1\le i\le n\}$ be the vertex sets of two copies of $\S_n$. 
    The graph $((\S_n\antimat\S_n-\{v_1, w_2\})\pivot v_2w_1)-\{v_2, w_1\}$ is isomorphic to $\S_{n-2}\mat\S_{n-2}$.
    
    (4) Let $V(\K_n)=\{v_i:1\le i\le n\}$ and $V(\S_n)=\{w_i:1\le i\le n\}$.
	The graph $(\K_n\antimat\S_n-w_1)*v_1-v_1$ is isomorphic to $\S_{n-1}\mat\K_{n-1}$.
	Thus, by (1), it contains a vertex-minor isomorphic to $\S_{n-2}\mat\S_{n-2}$.

    (5) Let $\{v_i:1\le i\le n\}$ and $\{w_i:1\le i\le n\}$ be the vertex sets of two copies of $\K_n$. 
    The graph $(\K_n\antimat\K_n-w_1)*v_1-v_1$ is isomorphic to $\S_{n-1}\mat\S_{n-1}$.
   \end{proof}
    
    From $H_1\tri H_2$
	with $H_1, H_2\in \{\S_n, \K_n\}$, 
	we can obtain a long induced path as a vertex-minor.
	So, if $n$ is sufficiently large, then this directly gives us $mP_{k+1}$ for some large $m$.
     
\begin{lemma}[Kwon and Oum~\cite{KwonO2014}]\label{lem:lengthonecase} 
Let $n$ be a positive integer.
\begin{enumerate}[(1)]
	\item $\S_{n}\tri\S_{n}$ is locally equivalent to $P_{2n}$.
	\item $\K_{n}\tri\S_{n}$ is locally equivalent to $P_{2n}$.
	\item If $n\ge 2$, then $\K_{n}\tri\K_{n}$ has a vertex-minor isomorphic to $P_{2n-2}$.
\end{enumerate}
\begin{proof}
 (1) and (2) are proved in \cite{KwonO2014}. To prove (3), 
 let $\{v_i:1\le i\le n\}$ and $\{w_i:1\le i\le n\}$ be the vertex sets of two copies of $\K_n$, where $v_i$ is adjacent to $w_j$ if and only if $i\ge j$. 
 Then $(\K_{n}\tri\K_{n}-w_1)*v_1-v_1$ is isomorphic to $\S_{n-1}\tri\K_{n-1}$. Thus, the result follows from (2).
\end{proof}
\end{lemma}
   
   We will prove the backward direction of Theorem~\ref{thm:main4} by induction on $k$.
   In the procedure, we find a vertex-minor containing a vertex set $S$ which induces a subgraph isomorphic to 
   $mH$ for some connected graph $H$ on $k$ vertices.
   Generally, we meet two situations: the cut-rank of $S$ is large or small.
   In the next lemma, we prove that if the cut-rank of $S$ is large, then 
   we can directly find a vertex-minor isomorphic to the disjoint union of many copies of some connected graph on $k+1$ vertices.
   If the cut-rank is small, then we will recursively find another such set after excluding $S$.

   \begin{lemma}\label{lem:largerankmatching}
     For positive integers $k$ and $n$,
     there exists a positive integer $m=f_1(k,n)$
     such that
     if a graph $G$ admits a set $W=\{w_1,\ldots,w_m\}$
     that is a clique or an independent set
     satisfying the following two properties,
     then $G$ has a vertex-minor isomorphic to
     $nH'$ for some connected graph $H'$ on $k+1$ vertices.
     \begin{enumerate}[(i)]
     \item    $G-W= m H$
     for some connected graph $H$ on $k$ vertices.
   \item 
     For some vertex $v$ of $H$
     and its copies $v_1$, $v_2$, $\ldots$, $v_m$  in $mH$,
     $v_i$ is adjacent to $w_j$ if and only if $i=j$.
     (In other words, the subgraph induced by $W\cup \{v_1,v_2,\ldots,v_m\}$ is isomorphic to $K_m\mat \S_m$ or $\S_m\mat \S_m$.)
     \end{enumerate}
   \end{lemma}
   \begin{proof}
     Let $H_i$ be the $i$-th copy of $H$ in $G-W$.
     We fix an isomorphism from $H$ to $H_i$ and isomorphisms between copies of $H$
     so that these isomorphisms are compatible.

     Assume that $m>2^{k-1}(m_1-1)$.
     For each $w_i$, there are at most $2^{k-1}$ possible sets of neighbors
     in $H_i$. So there exists a subset $W_1$ of $W$ with $\abs{W_1}= m_1$
     such that 
     the set of all neighbors of each $w_i\in W_1$ in $H_i$
     is identical up to isomorphisms between copies of $H$.
     
     Assume that $m_1\ge R(m_2;(2^{k-1})^2)$.
     For a vertex $w_i$ and $j\neq i$, 
     there are $2^{k-1}$ possible ways of having edges between
     the $j$-th copy of $H-v$ and $w_i$.
     By applying Ramsey's theorem,
     we deduce that there exists a subset $W_2\subseteq W_1$ of size $m_2$
     such that for all $i<j$ with $w_i, w_j\in W_2$,
     the set of all neighbors of $w_i$ in $H_j$
     is identical up to isomorphisms between copies of $H$
     and the set of all neighbors of $w_j$ in $H_i$
     is identical up to isomorphisms between copies of $H$.

     Assume that \[m_2\ge \max\left(\left\lceil\frac{(k+2)n-1}2\right\rceil+1,n+3\right).\]
     Suppose that there exist $i_1<i_2<i_3$ such that
     $w_{i_1},w_{i_2},w_{i_3}\in W_2$ and 
     there exists a vertex $u$ of $H$ so that
     exactly one of the copies of $u$ in $H_{i_1}$ and $H_{i_3}$
     is adjacent to $w_{i_2}$. 
     Then $G$ contains $\S_{m_2-1}\tri \S_{m_2-1}$ 
     or $\S_{m_2-1}\tri K_{m_2-1}$ as an induced subgraph.
     By Lemma~\ref{lem:lengthonecase}, 
     $G$ has a vertex-minor isomorphic to $P_{(k+2)n-1}$
     and therefore $G$ has $nP_{k+1}$ as a vertex-minor.

     Thus we may assume that there are no such $i_1<i_2<i_3$. 
     Since $m_2\ge 3$, 
     for all $i\neq j$ with $w_i, w_j\in W_2$,
     the set of all neighbors of $w_i$ in $H_j$
     is identical up to isomorphisms between copies of $H$.

     Suppose that $w_i\in W_2$ has no neighbors in $H_j$ when $j\neq i$
     and $w_j\in W_2$.
     If $W_2$ is an independent set, then
     clearly $G$ has an induced subgraph isomorphic to $m_2 H'$ for
     some connected graph $H'$ on $k+1$ vertices.
     If $W_2$ is a clique, then
     for some $w_i\in W_2$, $G*w_i$
     contains an induced subgraph isomorphic to $(m_2-1)H'$
     for some connected graph $H'$ on $k+1$ vertices.

     Thus, we may assume that $w_i\in W_2$ has at least one neighbor $u_j$
     in $H_j$ for some $j\neq i$ with $w_j\in W_2$.
     Let $G'=G\pivot w_iu_j-V(H_i)-V(H_j)-w_i-w_j$.
     If $W_2$ is an independent set, then
     $G'$ has an induced subgraph isomorphic to $(m_2-2)H'$ for
     some connected graph $H'$ on $k+1$ vertices.
     This is because, by (ii), in $G$, $v_\ell$
     is adjacent to $w_\ell$
     and non-adjacent to $w_i$ and $u_j$ 
     for all $\ell$ with $w_\ell\in W_2$, $\ell\neq i,j$.

     If $W_2$ is a clique, then
     let $w_{i_1}\in W_2\setminus\{w_i,w_j\}$
     and $G''=G'*w_{i_1}-V(H_{i_1})$.
     Then $G''$ contains an induced subgraph isomorphic to $(m_2-3)H'$
     for some connected graph $H'$ on $k+1$ vertices,
     again by (ii).

     So we can take \[
       f_1(k,n):= 2^{k-1}\left(R\bigl(
\max(\lceil {\textstyle \frac{(k+2)n-1}{2}}\rceil+1,n+3)     
;\allowbreak (2^{k-1})^2\bigr)-1\right)+1.\qedhere\]
   \end{proof}
     \begin{figure}
 \tikzstyle{v}=[circle,draw,fill=black!50,inner sep=0pt,minimum width=4pt]
  \centering
\begin{tikzpicture}[xscale=0.6,yscale=0.4]
      \foreach \x in {1,...,5} {
        \node [v]  (v\x) at(0,1.5*-\x){};
        \node [v]  (w\x) at (3,1.5*-\x){};
        \draw (0,1.5*-\x-0.5) node [left] {$v_\x$};
        \draw (3,1.5*-\x-0.5) node [right] {$w_\x$};
      
        \node [v]  (a\x) at (-1.5,1.5*-\x){};
        \node [v]  (b\x) at (-3,1.5*-\x){};
        \node [v]  (c\x) at (-4.5,1.5*-\x){};
     
      }
      \foreach \x in {1,...,5}
      {
      \draw(c\x)--(b\x)--(a\x)--(v\x);
		\draw (b\x) [in=150, out=30] to (v\x);
		}
            \foreach \x in {1,...,5} 
      \foreach \y in {1,...,5} {
       \ifnum \x=\y  \else      \draw (v\x)--(w\y);       \fi 
      }
   \draw (0,-9) node {$G$};
    \end{tikzpicture}
\qquad
\begin{tikzpicture}[xscale=0.6,yscale=0.4]
      \foreach \x in {3,...,5} {
        \node [v]  (v\x) at(0,1.5*-\x){};
        \node [v]  (w\x) at (3,1.5*-\x){};
        \draw (0,1.5*-\x-0.5) node [left] {$v_\x$};
        \draw (3,1.5*-\x-0.5) node [right] {$w_\x$};

      }
      
        \node [v]  (v1) at(0,-1.5){};
       \node [v]  (w2) at (3,-3){};
      \draw (0,-1.5-0.5) node [left] {$w_2$};
        \draw (3,-3-0.5) node [right] {$v_1$};
      \draw[dotted] (v1)--(w2);
      
      \foreach \x in {3,...,5}
      {
      \draw[dotted] (v1)--(w\x);
      \draw[dotted] (w2)--(v\x);
       \node [v]  (a\x) at (-1.5,1.5*-\x){};
        \node [v]  (b\x) at (-3,1.5*-\x){};
        \node [v]  (c\x) at (-4.5,1.5*-\x){};

      \draw(c\x)--(b\x)--(a\x)--(v\x);
		\draw (b\x) [in=150, out=30] to (v\x);
		}
     
      \foreach \x in {3,...,5} 
      \foreach \y in {3,...,5} {
       \ifnum \x=\y  \draw (v\x)--(w\y);       \fi    
      }
      \draw (0,-9) node {$G'$};
       
    \end{tikzpicture}
  \caption{Obtaining $G' =
        (G\pivot v_1w_2)-V(H_1)-V(H_2)-w_1-w_2$ from $G$ in the proof of Lemma~\ref{lem:largerankantimatching}.}
  \label{fig:g1tog2}
\end{figure}
   
   \begin{lemma}\label{lem:largerankantimatching}
     For positive integers $k$ and $n$,
     there exists a positive integer $m=f_2(k,n)$
     such that
     if a graph $G$ admits a set $W=\{w_1,\ldots,w_m\}$
     that is a clique or an independent set
     satisfying the following two properties,
     then $G$ has a vertex-minor isomorphic to
     $nH'$ for some connected graph $H'$ on $k+1$ vertices.
     \begin{enumerate}[(i)]
     \item    $G-W= m H$
     for some connected graph $H$ on $k$ vertices.
   \item 
     For some vertex $v$ of $H$
     and its copies $v_1$, $v_2$, $\ldots$, $v_m$  in $mH$,
     $v_i$ is adjacent to $w_j$ if and only if $i\neq j$.
     (In other words, the subgraph induced by $W\cup \{v_1,v_2,\ldots,v_m\}$ is isomorphic to $\S_m\antimat \S_m$ or $K_m\antimat \S_m$.)     
     \end{enumerate}
   \end{lemma}
   \begin{proof}
     Let $f_2(k,n):=f_1(k,n)+2$ for the function $f_1$ in Lemma~\ref{lem:largerankmatching}.
     Let $G'=(G\pivot v_1w_2)-V(H_1)-V(H_2)-w_1-w_2$
     where $H_1$, $H_2$ are the first and second copies of $H$.
     Then
     $G'-(W\setminus \{w_1,w_2\})$ is isomorphic to $(m-2)H$
     and 
     $G'$ satisfies the condition for Lemma~\ref{lem:largerankmatching}.
     See Figure~\ref{fig:g1tog2} for an illustration.
   \end{proof}
   
	\begin{lemma}\label{lem:largerank}
	For positive integers $k$ and $n$, there exists an integer $N:=N(k,n)$ with the following property. Let $H$ be a connected graph on $k$ vertices, and
	$G$ be a graph and $S\subseteq V(G)$ such that $G[S]$ is isomorphic to $qH$ for some integer $q$ and $\cutrk_G(S)\ge N$. 
	Then $G$ contains a vertex-minor isomorphic to $nH'$ for some connected graph $H'$ on $k+1$ vertices.
	\end{lemma}
\begin{proof}
  Let $f$ be the function defined in Theorem~\ref{thm:largebipartite}.
  Let $f_1$, $f_2$ be the functions defined in Lemmas~\ref{lem:largerankmatching} and \ref{lem:largerankantimatching}.
	We define that 
	\begin{align*}
          n_3(k,n)&:=\max(f_1(k,n),f_2(k,n), \lceil {\textstyle\frac{(k+2)n-1}{2}}\rceil ),\\
                    n_2(k,n)&:=
  \begin{cases}
    (k-1)n_3(k,n)+1  &\text{if }k>1,
    \\
    \max(n+2, \lceil (3n+1)/2\rceil) &\text{if }k=1,
  \end{cases}\\
               n_1(k,n)&:=R\left( n_2(k,n); 2 \right),\\
          N(k,n)&:=f(n_1(k,n)).
\end{align*}
	We shortly denote $n_1(k,n)$, $n_2(k,n)$, $n_3(k,n)$, $N(k,n)$ as $n_1$, $n_2$, $n_3$, $N$ respectively.
	
	Choose $B\subseteq V(G)\setminus S$ such that
	$\abs{B}=N$ and $\rank \left( A(G)[S, B] \right)=N$. 

	Observe that two distinct vertices in $B$ have distinct sets of neighbors in $S$.
	Since $N=f(n_1)$, by Theorem~\ref{thm:largebipartite}, there exist $A_1\subseteq S$ and $B_1\subseteq B$ with $\abs{A_1}=\abs{B_1}=n_1$ such that
	 $G[A_1, B_1]$ is isomorphic to $\S_{n_1}\mat\S_{n_1}$, $\S_{n_1}\tri \S_{n_1}$, or
     $\S_{n_1}\antimat\S_{n_1}$.

     Since $n_1=R(n_2;2)$, by Ramsey's theorem, there exists $B_2\subseteq B_1$ such that $\abs{B_2}= n_2$ and $B_2$ is a clique or an independent set.
     Let $A_2\subseteq A_1$ be the set of vertices matched with vertices in $B_2$ in the subgraph $G[A_1, B_1]$.
	Thus, $G[A_2, B_2]$ is isomorphic to 
    $\S_{n_2}\mat\S_{n_2}$, $\S_{n_2}\tri \S_{n_2}$, or
     $\S_{n_2}\antimat\S_{n_2}$.
     
      If $k=1$,     
     then by Lemma~\ref{lem:tomatching} or \ref{lem:lengthonecase},
     $G[A_2\cup B_2]$ contains a vertex-minor isomorphic to $\S_n\mat \S_n$, because $n_2\ge n+2$, $n_2\ge (3n+1)/2$,
     and $P_{3n-1}$ has $\S_n\mat\S_n$ as an induced subgraph.
   So, we may assume that $k\ge 2$.

 	Observe that $H$ has a vertex $v$ such that $A_2$ has at least $\lceil n_2/k\rceil=n_3$ copies of $v$.
     Let $A_3$ be a set of $n_3$ copies of $v$ in $A_2$, and $B_3\subseteq B_2$ be the set of vertices matched with vertices in $A_3$ in the subgraph $G[A_2, B_2]$.
     Let $\mathcal{C}$ be the set of components of $G[S]$ containing a vertex in $A_3$.
     Clearly, we have 
     \begin{itemize}
     \item $\abs{\mathcal{C}}= n_3$, 
     \item $G[A_3, B_3]$ is isomorphic to
    $\S_{n_3}\mat\S_{n_3}$, $\S_{n_3}\tri \S_{n_3}$, or
     $\S_{n_3}\antimat\S_{n_3}$, 
     \item $A_3$ is an independent set,
     \item $B_3$ is a clique or an independent set.
     \end{itemize}

     If $G[A_3, B_3]$ is isomorphic to $\S_{n_3}\tri \S_{n_3}$, 
     then $G[A_3\cup B_3]$ is isomorphic to $\S_{n_3}\tri \S_{n_3}$ or $\S_{n_3}\tri K_{n_3}$, 
     and thus by Lemma~\ref{lem:lengthonecase}, 
     it is locally equivalent to $P_{2n_3}$.
     As $2n_3\ge (k+2)n-1$,  
     $P_{2n_3}$ contains an induced subgraph isomorphic to $nP_{k+1}$.
     Therefore, we may assume $G[A_3, B_3]$ is isomorphic to 
     $\S_{n_3}\mat\S_{n_3}$ or
     $\S_{n_3}\antimat\S_{n_3}$.
     By Lemmas~\ref{lem:largerankmatching} and \ref{lem:largerankantimatching},
     we deduce that $G$ has a vertex-minor isomorphic to $nH'$
     for some connected graph $H'$ on $k+1$ vertices.
\end{proof}

From now on, our main focus is to deal with the case that 
the cut-rank of $S$ is small, where $S$ is the vertex set inducing the disjoint union of many copies of a connected graph $H$.

\begin{lemma}\label{lem:rowtwins}
  Let $k$ and $n$ be positive integers and
  let $\ell=k2^{k(N(k,n)-1)}+1$ for the function $N$ in Lemma~\ref{lem:largerank}.
  Let $H$ be a connected graph on $k$ vertices.
  If $G$ has an induced subgraph isomorphic to $\ell H$,
  then at least one of the following holds.
  \begin{enumerate}[(i)]
  \item $G$ has a vertex-minor isomorphic to $nH'$ for some connected graph $H'$ on
    $k+1$ vertices.
  \item There exists $A\subseteq V(G)$ such
    that $G[A]$ is isomorphic to $(k+1)H$
    and for each vertex of $H$, its copies in $G[A]$
    have the same set of neighbors in $V(G)\setminus A$.
  \end{enumerate}
\end{lemma}
	\begin{proof}
	Let $S\subseteq V(G)$ be a vertex set such that
	$G[S]$ is isomorphic to $\ell H$.

	If $\cutrk_{G}(S)\ge N(k,n)$, then by Lemma~\ref{lem:largerank}, 
	$G$ contains a vertex-minor isomorphic to $nH'$ for some connected graph $H'$ on $k+1$ vertices. 
	Therefore, we may assume that $\cutrk_{G}(S)<N(k,n)$.

        Let $V(H)=\{z_1, z_2, \ldots, z_k\}$.
	For each $i\in \{1, 2, \ldots, k\}$, let $Z_i$ be the set of all copies of $z_
        i$ in $S$.
	Since  $\cutrk_{G}(S)<N(k,n)$,
        \[ \rank A(G)[Z_i, V(G)\setminus S]\le N(k,n)-1\] for each $i\in \{1, 2, \ldots, k\}$ and so 
        $A(G)[Z_i, V(G)\setminus S]$ has at most $2^{N(k,n)-1}$  distinct rows because it is a $0$-$1$ matrix.
	In other words, \[\abs{\{N_G(v)\cap (V(G)\setminus S): v\in Z_i\}}\le 2^{N(k,n)-1}\] for each $1\le i\le k$.

        Since $\lceil\ell/2^{k(N(k,n)-1)}\rceil \ge k+1$, by the pigeon-hole principle,
        there exists a set $\mathcal C$ of at least $k+1$ components of $G[S]$ such that
	for each $i\in \{1, 2, \ldots, k\}$, 
	vertices in $Z_i\cap (\bigcup_{C\in \mathcal C}V(C))$ have the same set of neighbors in $V(G)\setminus S$.
	It implies (ii).
	\end{proof}

\begin{lemma}\label{lem:samecopy}
  Let $k$, $n$ be positive integers.
  If a graph has more than $(n-1)2^{\binom{k+1}{2}}$ components having exactly $k+1$ vertices, then it contains an induced subgraph isomorphic to $nH$ for some connected graph $H$ on $k+1$ vertices.
\end{lemma}
\begin{proof}
  The number of non-isomorphic graphs on $k+1$ vertices is at most $2^{\binom{k+1}{2}}$.
    By the pigeon-hole principle, 
        at least $n$ components are pairwise isomorphic.
\end{proof}
We will use the following lemma under the condition that $t=k$ but 
we prove a stronger statement for the convenience of the proof. 
\begin{lemma}\label{lem:sameneighbors}
  Let $k$, $t$ be integers such that $1\le t\le k$.
  Let $H$ be a connected graph on $k$ vertices.
  Let $G$ be a graph such that every component has more than $k$ vertices
  and $G$ contains $(t+1)H$ as an induced subgraph.
  If
  \begin{itemize}
  \item   for each vertex of $H$, their copies in $(t+1)H$
  have the same set of neighbors in $V(G)\setminus V((t+1)H)$
  and
  \item each component of $(t+1)H$ has at most $t$ vertices having a neighbor
  in $V(G)\setminus V((t+1)H)$,
  \end{itemize}
  then there exist a graph $G'$ locally equivalent to $G$,
  disjoint subsets $S$, $T$ of $V(G')$ and a vertex $v$ in $S$ such that
  \begin{enumerate}[(i)]
  \item   $G'[S]$ is a connected graph on $k+1$ vertices,
  \item   $\abs{T}\le t(k+1)$, and 
    \item  $G'[S\setminus\{v\}]$ is a component of $G'-(T\cup\{v\})$.
  \end{enumerate}
\end{lemma}
\begin{proof}
	Let $A\subseteq V(G)$ such that $G[A]$ is isomorphic to $(t+1)H$.
	Let $\cC:=\{C_1, C_2, \ldots, C_{t+1}\}$ be the set of components of $G[A]$, and let $V(H)=\{z_1, z_2, \ldots, z_k\}$.
	For each $i\in \{1, 2, \ldots, k\}$, let $Z_i$ be the set of all copies of $z_i$ in $A$.
	Let $U_i$ be the set of neighbors of vertices of $Z_i$ on $V(G)\setminus A$ in $G$, that is, $U_i=N_G(r)\cap (V(G)\setminus A)$ for $r\in Z_i$.
	Let $X\subseteq \{1, 2, \ldots, k\}$ be the set of integers $i$ such that $U_i$ is non-empty. By the assumption $\abs{X}\le t$.
        Since each component of $G$ has more than $k$ vertices,
        we have $\abs{X}>0$.
	Without loss of generality, we may assume $X=\{1, \ldots, \abs{X}\}$.

	We proceed by induction on $t$.
        
	If $t=1$, then let $x\in Z_1\cap V(C_1)$ and $y\in U_1$. 
	We obtain a new graph from $G$ by removing vertices of $V(C_1)\setminus \{x\}$ and pivoting $xy$.
	Note that the set of neighbors of $x$ in $G-(V(C_1)\setminus \{x\})$ is exactly $U_1$.
	Thus,  after pivoting $xy$, all edges between the vertex $z$ in $Z_1\cap V(C_2)$ and $U_1\setminus \{y\}$ are removed 
	and $z$ has exactly one neighbor $x$ on $V(G)\setminus V(C_2)$.
	Therefore, $(G', S, T, v)=(G\pivot xy, V(C_2)\cup \{x\}, (V(C_1)\setminus \{x\})\cup \{y\}, x)$ is a required tuple.
	
	Now we assume that $t\ge 2$.
        We may assume that $\abs{X}=t$ by the induction hypothesis.

	Let $x\in Z_1\cap V(C_1)$ and $y\in U_1$. 
	We obtain $G_1$ from $G$ by removing vertices of $V(C_1)\setminus \{x\}$ and pivoting $xy$.
        Let $A_1=A\setminus V(C_1)$.
	Note that in $G$,  the set of neighbors of $x$
        in $V(G)\setminus V(C_1)$ is exactly $U_1$.
	Thus, 
	\begin{itemize}
	\item the adjacency relations between two vertices in $A_1$ do not change by pivoting $xy$, 
	\item all edges between $Z_1\setminus \{x\}$ and $U_1\setminus \{y\}$ are removed by pivoting $xy$.
	\end{itemize}
	Furthermore, as vertices in each $Z_i$ have the same set of neighbors on $V(G)\setminus A$ in $G$, $G_1$ has the following properties.
	\begin{itemize}
	\item For all $i'\in \{2, \ldots, t\}$, two vertices in $Z_{i'}\cap A_1$ have the same set of neighbors in $V(G_1)\setminus A_1$.
	\item If $t<k$, then  for $i'\in \{t+1, \ldots, k\}$, vertices in $Z_{i'}\cap  A_1$ have no neighbors in  $V(G_1)\setminus A_1$.
	\end{itemize}

	If vertices in $Z_j\cap A_1$ have no neighbors on $V(G_1)\setminus (A_1\cup\{x,y\})$ for all $2\le j\le k$ in $G_1$, then 
        $(G', S, T, v)=(G\pivot xy, V(C_2)\cup \{x\}, (V(C_1)\setminus \{x\})\cup \{y\}, x)$ is a required tuple.
	Thus, we may assume that there is $j\in \{2, \ldots, k\}$ such that 
	vertices in $Z_j\cap A_1$ have a neighbor on $V(G_1)\setminus (A_1\cup\{x,y\})$ in $G_1$.

	Note that $G_1-\{x,y\}$ contains an induced subgraph isomorphic to $tH$ on the vertex set $A_1$
  such that  
  \begin{itemize}
  \item for each vertex of $H$, their copies in $tH$
    have the same set of neighbors in $V(G_1-\{x,y\})\setminus A_1$,
   \item each component of $tH$ has at least one and less than $t$ vertices having a neighbor in $V(G_1-\{x,y\})\setminus A_1$.
    \end{itemize}
    By the induction hypothesis, $G_1-\{x,y\}$ has the tuple $(G',S,T,v)$.
    Let $G''$ be the graph locally equivalent to $G$
    such that $G''-V(C_1)-y=G'$.
    Then $(G'',S,T\cup V(C_1)\cup \{y\}, v)$ is a required tuple for $G$.
	\end{proof}

	We prove the main proposition.
	
	\begin{proposition}\label{prop:rkbrittleconverse}
	For positive integers $k$ and $n$, there exists an integer $\ell=\ell(k,n)$ such that 	
	every graph with rank $k$-brittleness more than $\ell$ contains a vertex-minor isomorphic to 
	$nH$ for some connected graph $H$ on $k+1$ vertices.
	\end{proposition}
\begin{proof}
Let $f, N$ be the functions defined in Theorem~\ref{thm:largebipartite} and Lemma~\ref{lem:largerank}, respectively.
	We define
	\begin{itemize}
	\item $\ell_2(1,n):=\max (n+2, \lceil (3n+1)/2\rceil)$, 
	\item $\ell_1(1,n):=R(\ell_2(1,n);4)$,
	\item $\ell(1,n):=f(\ell_1(1,n))-1$, 
	\end{itemize}
	and for $k\ge 2$, let
	\begin{itemize}
	\item $\ell_3(k,n):=k2^{k(N(k,n)-1)}+1$,
	\item $\ell_2(k,n):= \max\left((k+2)n, 2^{\binom{k+1}{2}}(n-1)+2\right)$,
	\item $\ell_1(k,n):=R(\ell_2(k,n);2^{k+1})$,
	\item $\ell(k,n):=\ell(k-1, \ell_3(k,n))+(k+1)^2 (\ell_1(k,n)-1)$.
	\end{itemize}
	We will prove the statement by induction on $k$. 
	We shortly denote $\ell_1(k,n)$, $\ell_2(k,n)$, $\ell_3(k,n)$, $\ell(k,n)$ as $\ell_1$, $\ell_2$, $\ell_3$, $\ell$,  respectively.

Let us first consider the case that $k=1$.
	Suppose $G$ has rank $1$-brittleness more than $\ell$.	
	Then, there exists a vertex set $A$ such that $\cutrk_G(A)> \ell$.
	Choose $A_1\subseteq A$ and $B_1\subseteq V(G)\setminus A$
        such that $\abs{A_1}=\abs{B_1}=\ell+1$ and 
	$\rank \left( A(G)[A_1,B_1]\right) =\ell+1$. 
	Note that two vertices in $B_1$ have distinct neighbors on $A_1$.
	Since $\ell+1=f(\ell_1)$, by Theorem~\ref{thm:largebipartite}, 
	there exist $A_2\subseteq A_1$ and $B_2\subseteq B_1$ with $\abs{A_2}=\abs{B_2}=\ell_1$ such that
       $G[A_2, B_2]$ is isomorphic to $\S_{\ell_1}\mat\S_{\ell_1}$, $\S_{\ell_1}\tri \S_{\ell_1}$, or
     $\S_{\ell_1}\antimat\S_{\ell_1}$.

	As $\ell_1= R(\ell_2;4)$, by Ramsey's theorem,
	there exist $A_3\subseteq A_2$ and $B_3\subseteq B_2$ such that 
	\begin{itemize}
	\item $G[A_3, B_3]$ is isomorphic to $\S_{\ell_2}\mat\S_{\ell_2}$, $\S_{\ell_2}\tri \S_{\ell_2}$, or
     $\S_{\ell_2}\antimat\S_{\ell_2}$, and
     \item each of $A_3$ and $B_3$ is a clique or an independent set.
	\end{itemize}
	If $G[A_3, B_3]$ is isomorphic to $\S_{\ell_2}\tri \S_{\ell_2}$, then by  Lemma~\ref{lem:lengthonecase},
	$G[A_3\cup B_3]$ contains a vertex-minor isomorphic to $P_{2\ell_2-2}$. 
	As $2\ell_2-2\ge 2(\frac{3n+1}{2})-2\ge 3n-1$, $P_{2\ell_2-2}$ contains an induced subgraph isomorphic to $nK_2$.
	Therefore we may assume that $G[A_3, B_3]$ is isomorphic to 
	$\S_{\ell_2}\mat \S_{\ell_2}$ or $\S_{\ell_2}\antimat\S_{\ell_2}$.
	Because $\ell_2\ge n+2$, by Lemma~\ref{lem:tomatching}, $G$ contains a vertex-minor isomorphic to $\S_{n}\mat\S_{n}$,
	which is isomorphic to $nK_2$, as required.

	Now, we prove for $k\ge 2$.
	Suppose $G$ has rank $k$-brittleness more than $\ell$.
	Among all graphs $G'$ locally equivalent to $G$,
        choose $G'$ admitting a sequence of $m+1$ tuples
	\[(S_0, T_0), (S_1, T_1, v_1), (S_2, T_2, v_2), \ldots, (S_m, T_m, v_m)\]
        with the maximum $m$
        such that
	\begin{itemize}
	\item $S_0=T_0=\emptyset$, 
	\item $S_1, S_2, \ldots, S_m, T_1, T_2, \ldots, T_m$ are pairwise disjoint subsets of $V(G')$,
	\item for each $i\in \{1, 2, \ldots, m\}$, 
	\begin{itemize}
	\item $\abs{S_i}=k+1$ and $G'[S_i]$ is connected, 
	\item $\abs{T_i}\le k(k+1)$, 
	\item $v_i\in S_i$, 
	\item no vertex in $S_i\setminus \{v_i\}$
          has a neighbor in $V(G')\setminus (\bigcup_{0\le j\le i}(S_j\cup T_j))$.
	\end{itemize}
      \end{itemize}

      Such a graph $G'$ exists trivially because $(S_0,T_0)$ is a valid sequence
      for $G$ and so $m\ge 0$.
	Suppose that $m<\ell_1$.
	Let $G_1:=G'-(\bigcup_{0\le j\le m}(S_j\cup T_j))$.
	Since $G'$ is locally equivalent to $G$, $\rkbrit_k(G')=\rkbrit_k(G)$, and
	therefore, 
	\[\rkbrit_k(G')=\rkbrit_k(G)>\ell(k-1, \ell_3)+(k+1)^2 (\ell_1-1).\] 
	As $\abs{\bigcup_{0\le j\le m}(S_j\cup T_j)}\le (k+1)^2 m\le (k+1)^2 (\ell_1-1)$,
	by Proposition~\ref{prop:vmbrittle}, 
	 we have that 
	$\rkbrit_k(G_1)> \ell(k-1, \ell_3)$.
	Let $G_2$ be the graph obtained from $G_1$ by removing all components having at most $k$ vertices.
	It is not difficult to observe that $\rkbrit_k(G_2)=\rkbrit_k(G_1)$.

	As $\rkbrit_{k-1}(G_2)\ge \rkbrit_k(G_2)$, by the induction hypothesis, 
	$G_2$ contains a vertex-minor isomorphic to $\ell_3 F$ for some connected graph $F$ on $k$ vertices.
	Thus, there exist a graph $G_3$ locally equivalent to $G_2$ and a vertex subset $A$ of $G_3$ such that
	$G_3[A]$ is isomorphic to $\ell_3F$.

	Note that $\ell_3= k2^{k(N(k,n)-1)}+1$.
	So, by Lemma~\ref{lem:rowtwins}, either
	\begin{enumerate}[(1)]
	\item $G_3$ contains a vertex-minor isomorphic to $nH$ for some connected graph $H$ on $k+1$ vertices, 
	or 
	\item there exists $A'\subseteq V(G_3)$ such
    that $G_3[A']$ is isomorphic to $(k+1)F$
    and for each vertex of $F$, its copies in $G_3[A']$
    have the same set of neighbors in $V(G_3)\setminus A'$.
    \end{enumerate}
    We may assume that $(2)$ holds.
    Since $G_3$ is locally equivalent to $G_2$, 
    every component of $G_3$ has more than $k$ vertices.
    By Lemma~\ref{lem:sameneighbors} (with $t:=k$), 
    there exist a graph $G_4$ locally equivalent to $G_3$,
  disjoint subsets $S$, $T$ of $V(G_4)$, and a vertex $v$ in $S$ such that
  \begin{enumerate}[(i)]
  \item   $G_4[S]$ is a connected graph on $k+1$ vertices,
  \item   $\abs{T}\le k(k+1)$, and 
    \item  $G_4[S\setminus\{v\}]$ is a component of $G_4-(T\cup\{v\})$.
  \end{enumerate}
  
  In $G'$, no vertex in $S_i\setminus\{v_i\}$ has a neighbor
  in $V(G')\setminus (\bigcup_{0\le j\le m}(S_j\cup T_j))$.
  Let $G''$ be the graph obtained from $G'$ by applying the same sequence of
  local complementations needed to obtain $G_4$ from $G_2$.
  Since $G_2$ has no vertex in $\bigcup_{0\le j\le m} (S_j\cup T_j)$
  and at most one vertex of $G'[S_i]$ has a neighbor in $V(G')\setminus\bigcup_{0\le j\le m} (S_j\cup T_j)$, we deduce that  
  $G''[S_i]=G'[S_i]$ for all $i\in \{1,2,\ldots,m\}$.
  Therefore, $G''$ admits the sequence
  $(S_0,T_0)$, $(S_1,T_1,v_1)$, $\ldots$, $(S_m,T_m,v_m)$,
  $(S,T,v)$, contradicting the assumption on the choice of $G'$ with the maximum $m$.
  Thus $m\ge \ell_1$.
  \smallskip

  In $G'$, for $i,j\in \{1, 2, \ldots, \ell_1\}$ with $i<j$, $v_i$ may have neighbors on $S_j$, but $v_j$ has no neighbors on $S_i\setminus \{v_i\}$.
  Let $s_{i,1}, s_{i,2},\ldots,s_{i,k}$ be the vertices in $S_i\setminus\{v_i\}$ for each~$i$.
  
  We construct a complete graph on the vertex set $\{w_1, w_2, \ldots, w_{\ell_1}\}$, 
	 and for $i,j\in \{1, 2, \ldots, \ell_1\}$ with $i<j$, 
	 we color the edge $w_iw_j$ by one of $2^{k+1}$ colors, depending on the adjacency relation between 
	 $v_i$ and $S_j$.
	 As $\ell_1= R(\ell_2; 2^{k+1})$, 
	 there exists a subset $I\subseteq \{1, 2, \ldots, \ell_1\}$ such that 
	 $\abs{I}=\ell_2$ and
         edges between two vertices in $\{w_i:i\in I\}$ are monochromatic.
         This also implies that $\{v_i:i\in I\}$ is a clique or an independent set.
         Let  $i_1<i_2<\cdots<i_{\ell_2}$ be the elements of $I$.

	 For some $i, j\in I$ with $i<j$, if $v_i$ is adjacent to $s_{j,j'}$ for some $j'$, 
	 then
         for all $i,j\in I$ with $i\neq j$,
         $v_i$ is adjacent to $s_{j,j'}$ if and only if $i<j$.
         By taking vertices $v_{i_1}$, $v_{i_3}$, $\ldots$, $v_{i_{2\lfloor \ell_2/2\rfloor-1}}$
         and $s_{i_2,j'}$, $s_{i_4,j'}$, $\ldots$, $s_{i_{2\lfloor \ell_2/2\rfloor},j'}$, 
         we obtain an induced subgraph of $G'$
         isomorphic to either $\S_{\lfloor \ell_2/2\rfloor}\tri\S_{\lfloor \ell_2/2\rfloor}$ or $\S_{\lfloor \ell_2/2\rfloor}\tri K_{\lfloor \ell_2/2\rfloor}$.
	By Lemma~\ref{lem:lengthonecase}, 
	$G'$ contains a vertex-minor isomorphic to $P_{\ell_2-1}$. 
	As $\ell_2-1\ge (k+2)n-1$, 
	 $P_{\ell_2-1}$ contains an induced subgraph isomorphic to $nP_{k+1}$.
	 Thus, $G$ contains a vertex-minor isomorphic to $nP_{k+1}$.
	 Therefore we may assume that 
	 for $i, j\in I$ with $i<j$, $v_i$ has no neighbors in $S_j\setminus \{v_j\}$. 

         If $\{v_i:i\in I\}$ is independent in $G'$,
         then
         $G'[\bigcup_{i\in I} S_i]$ is the disjoint union of
         $\ell_2$ connected graphs, each having exactly $k+1$ vertices.
         Since $\ell_2>2^{\binom{k+1}{2}}(n-1)$, by Lemma~\ref{lem:samecopy},
         $G$ contains a vertex-minor isomorphic to $nH$ for some connected graph $H$ on $k+1$ vertices.

         If $\{v_i:i\in I\}$ is a clique in $G'$,
         then
         let $i'\in I$ and let $G''= G'*v_{i'}$.
         Then 
         $G''[\bigcup_{i\in I,i\neq i'} S_i]$ is the disjoint union of
         $\ell_2-1$ connected graphs, each having exactly $k+1$ vertices.
         Since $\ell_2-1>2^{\binom{k+1}{2}}(n-1)$, by Lemma~\ref{lem:samecopy},
         $G$ contains a vertex-minor isomorphic to $nH$ for some connected graph $H$ on $k+1$ vertices.
       \end{proof}

       Here is the proof of Theorem~\ref{thm:main4}.
	Let $\mathcal{C}$ be a vertex-minor ideal.
	Suppose $\cC$ is rank $k$-scattered, that is, 
	there exists an integer $\ell$ such that 
        every graph $G\in \mathcal{C}$ has rank $k$-brittleness at most $\ell$.
	Then by (3) of Lemma~\ref{lem:edgeforward1}, for every connected graph $H$ on $k+1$ vertices, 
	$\mathcal{C}$ does not contain $(2\ell+1)H$.

        For the converse, suppose that for every connected graph $H$ on
        $k+1$ vertices, there exists $n_H$ such that $n_H H\notin \cC$.
        Since there are only finitely many non-isomorphic graphs on
        $k+1$ vertices, there exists the maximum $n$ among all $n_H$.
        Then $nH\notin \cC$ for all connected graphs $H$ on $k+1$ vertices.
        By Proposition~\ref{prop:rkbrittleconverse},
        all graphs in $\cC$ have rank $k$-brittleness at most $\ell(k,n)$.

\section{Comparisons}\label{sec:parameter}

In this section, we compare our concepts
with existing concepts on graphs.
See Figure~\ref{fig:relation} for the relations that we are going to prove.

\begin{figure}
        \centering
        \begin{tikzcd}[column sep=small]        
        && \substack{\text{bounded}\\\text{\emph{rank-width}}}
        \\
        &\substack{\text{bounded}\\\text{\emph{tree-width}}}
        \arrow [ur]
        & \substack{\text{bounded}\\\text{\emph{linear rank-width}}}
        \arrow[u]
        \\
        &\substack{\text{bounded}\\\text{\emph{path-width}}}
        \arrow [ur] \arrow[u]
        & \substack{\text{bounded}\\\text{\emph{rank-depth (shrub-depth)}}}
        \arrow[u]
        &
        \substack{\text{bounded}\\\text{\emph{modular-width}}}
        \arrow[luu]
        \\       
        &\substack{\text{bounded}\\\text{\emph{tree-depth}}}
        \arrow [ur]\arrow[u]
        & 
        \substack{\text{\emph{rank $k$-scattered}} \\ \text{for some $k$}}      \arrow[u]
        &
        \\
        \substack{\text{\emph{vertex $k$-scattered}} \\ \text{for some $k$}}
        \arrow[r,red,shift left=2,"k'\leftarrow k+1" ]
        &
        \substack{\text{\emph{matching $k$-scattered}} \\ \text{for some $k$}}
        \arrow [ur,red,"k'\leftarrow k"  description]
        \arrow [l,dashed, shift left=1]
        \arrow [u]
        \\
        &&
        \substack{\text{\emph{rank $1$-scattered}}\\}
        \arrow[r,Leftrightarrow]\arrow[uuur]\arrow[uu]
        &\substack{\text{bounded} \\ \text{\emph{neighborhood diversity}}}
        \\
        \substack{\text{\emph{edge $k$-scattered}} \\ \text{for some $k$}}
        \arrow[uur,red,"k'\leftarrow k" description]
        \arrow[uu,red, "k'\leftarrow \binom{k}{2}" description]        
        &
        \substack{\text{\emph{matching $1$-scattered}} \\}
        \arrow[uu]\arrow[ru]\arrow[r,Leftrightarrow]
        &
        \substack{\text{bounded} \\ \text{\emph{vertex cover number}}}
        \end{tikzcd}
        \caption{Comparing graph classes. 
        An arrow from $A$ to $B$ means that a class with the property $A$ satisfies the property $B$.
        A red solid arrow from $A$ to $B$ with the condition $k'\leftarrow f(k)$ implies that if a class has the property $A$ with $k$, then it has the property B with $k:=k'$.
        A dashed arrow from $A$ to $B$ means that if a class has the property $A$ with $k$, then it has the property B with $k:=k'$
        but we do not have a function for $k'$ depending only on $k$.}
        \label{fig:relation}
\end{figure}

\subsection{Vertex cover number and matching $1$-scatteredness}
A set $S$ of vertices in $G$ is a \emph{vertex cover} of $G$ if $G-S$ has no edges.
Let $\cover(G)$ denote the minimum size of a vertex cover of a graph $G$, 
which we call the \emph{vertex cover number} of $G$.
\begin{proposition}\label{prop:vcbrit}
        A class of graphs has bounded vertex cover number 
        if and only if it is matching $1$-scattered.                
\end{proposition}
\begin{proof}
        We claim that        
        \[ \mbrit_1(G)\le \cover(G)\le 4\mbrit_1(G).               \]
        It is easy to see that $\mbrit_1(G)\le \cover(G)$ because $G$ has no matching of size larger than $\cover(G)$. 

        Let us assume that $\mbrit_1(G)\le m$. 
        Let $M$ be a maximum matching of $G$.
        If $\abs{M}>2m$, then 
        by the probabilistic argument, there is a subset $I$ of $V(G)$ such that at least half of the edges in $M$ joins a vertex in $I$ to a vertex in $V(G)\setminus I$, contradicting the assumption that 
        $\mbrit_1(G)\le m$.
        So $\abs{M}\le 2m$. Then the set of all ends of edges in $M$ is a vertex cover of size at most $4m$.
\end{proof}

\begin{proposition}
        There is a matching $1$-scattered class of graphs that is not edge $k$-scattered for any integer~$k$.
\end{proposition}
\begin{proof}
        The graph $K_{1,n}$ has matching $1$-brittleness $1$, while it has edge $k$-brittleness at least $n-k+1$ by Lemma~\ref{lem:edgeforward2}.
\end{proof}
\subsection{Neighborhood diversity and rank $1$-scatteredness}
        The neighborhood diversity was introduced by Lampis~\cite{Lampis2012}.
        Two vertices $v$ and $w$ in a graph $G$ are \emph{twins} if $v$ and $w$ have the same set of neighbors in $V(G)\setminus\{v,w\}$. 
        The \emph{neighborhood diversity} of a graph $G$
        is the minimum $t$ such that there is a partition of the vertex set of $G$ into at most $t$ sets, each of which is a set of pairwise twins.
	\begin{proposition}\label{prop:diversity}
        A class of graphs has bounded neighborhood diversity 
        if and only if
        it is rank $1$-scattered.
	\end{proposition}
        \begin{proof}
                For the forward direction, we claim that the rank $1$-brittleness of a graph is less than or equal to its neighborhood diversity.
                Let $G$ be a graph of neighborhood diversity at most $t$.
                For a set $A$ of vertices, let $M_A$ be the $A\times (V(G)
                \setminus A)$ submatrix of the adjacency matrix of $G$ over the binary field so that $\rank M_A=\rho_G(A)$.
                Then $M_A$ has at most $t$ distinct rows and so $\cutrk_G(A)\le t$ for all $A\subseteq V(G)$.
                It implies that $\rkbrit_1(G)\le t$.
                
                The backward direction is implied by the lemma of 
                Nguyen and Oum~\cite[Lemma 5.3]{NO2019}, showing that
                if $\rho_G(X)\le n$ for all $X\subseteq V(G)$, then 
                the neighborhood diversity is at most $2^{2n+2}$. 
	\end{proof}
        Lampis~\cite[Lemma 2]{Lampis2012} shows that if  $G$ has a vertex cover of size $t$, then its neighborhood diversity is at most $2^t+t$.
        \begin{proposition}
                There is a class of graphs of neighborhood diversity $1$
                that has unbounded tree-width.
        \end{proposition}
        \begin{proof}
                The complete graph $K_n$ has neighborhood diversity $1$
                and yet its tree-width is $n-1$.
        \end{proof}

        \subsection{Modular-width} 
        The modular-width of a graph was defined by Gajarsk\'{y}, Lampis, and Ordyniak~\cite{GLO2013}.
        We remark that this modular-width is different 
        from the modular-width defined by Rao~\cite{Rao2008}. 
       
        A \emph{module} of a graph $G$ is a set $M$ of vertices such that 
        no vertex in $V(G)\setminus M$ has both a neighbor and a non-neighbor in $M$.
        A module $M$ is \emph{trivial} if $\abs{M}\le 1$ or $M=V(G)$. 
        A graph is \emph{prime} if it has no non-trivial modules.
        
        For a positive integer $k$, let $\mathcal M_k$ be the smallest class of graphs having the following four properties:
        \begin{enumerate}
           \item $\mathcal M_k$ contains all graphs of at most $1$ vertex.
           \item If $G$ and $H$ are in $\mathcal M_k$, then so is the disjoint union of $G$ and $H$.
           \item If $G$ and $H$ are in $\mathcal M_k$, then 
           so is the complete join of $G$ and $H$, that is the graph obtained from the disjoint union of $G$ and $H$ by adding edges between all pairs of vertices in $V(G)$ and $V(H)$.
           \item If $G_1$, $G_2$, $\ldots$, $G_m$ are graphs in $\mathcal M_k$ for some 
           $m\le k$ and $G$ is a graph on the vertex set $\{v_1,v_2,\ldots,v_m\}$, then 
           $\mathcal M_k$ contains 
           the graph obtained from $G$ by substituting $v_i$ with $G_i$ for all $1\le i\le m$. 
        \end{enumerate}
        The \emph{modular-width} of a graph $G$, denoted by $\mw(G)$, is the minimum positive integer $k$ such that $G\in \mathcal M_k$.

        We will use the fact that if every prime induced subgraph of a graph $G$ has at most $k$ vertices, then 
        the modular-width of $G$ is at most $k$.

        \begin{proposition}\label{prop:modularwidth}
               Every rank $1$-scattered class of graphs has bounded modular-width.
        \end{proposition}
        \begin{proof}

            We claim that if $m=\rkbrit_1(G)$, then 
            every prime induced subgraph of $G$ has less than $R(f(m+2); 2)$ vertices, where $f$ is the function in Theorem~\ref{thm:largebipartite}.
            Suppose for contradiction that a prime induced subgraph $H$ of $G$ has at least $R(f(m+2); 2)$ vertices.
            Then by Ramsey's theorem, $H$ has a clique or an independent set $A$ of size $f(m+2)$.
            For two vertices $v$, $w$ in $A$, 
            since $\{v,w\}$ is not a module of $H$, 
            $N_H(v)\setminus A\neq N_H(w)\setminus A$.
            So, by Theorem~\ref{thm:largebipartite}, 
            $H[A, V(H)\setminus A]$ contains an induced subgraph isomorphic to $\S_{m+2}\mat \S_{m+2}$, $\S_{m+2}\tri \S_{m+2}$, or $\S_{m+2}\antimat\S_{m+2}$.
            It implies that the matrix $A_H[A, V(H)\setminus A]$ has rank at least $m+1$, and therefore 
            $\cutrk_G(A)\ge m+1$, contradicting 
            the assumption that $G$ has rank $1$-brittleness $m$.
            Thus, every prime induced subgraph of $G$ has less than $R(f(m+2); 2)$ vertices, 
            and so $G$ has modular-width less than  $R(f(m+2); 2)$.
        \end{proof}
        \begin{proposition}
               There is a rank $2$-scattered class of graphs 
               having unbounded modular-width.
        \end{proposition}
        \begin{proof}
               It is easy to see that $K_n\mat K_n$ is prime if $n\ge 3$, and thus it has modular-width $2n$ if $n\ge 3$.
               But $\rkbrit_2(K_n\mat K_n)\le 2$
               and so $\{K_n\mat K_n:n\ge 3\}$ is rank $2$-scattered.
        \end{proof} 
        
\subsection{Edge $k$-scatteredness}
\begin{proposition}\label{prop:edge-vm}
        \begin{enumerate}[(1)]
                \item Every edge $k$-scattered class of graphs is vertex $\binom{k}{2}$-scattered
                and matching $k$-scattered.
                \item For every integer $k>1$, there exists an edge $k$-scattered class of graphs that is neither vertex $(\binom{k}{2}-1)$-scattered nor matching $(k-1)$-scattered.
        \end{enumerate}
\end{proposition}
\begin{proof}   
        (1)    
        We claim that \[\vbrit_{\binom{k}{2}}(G)\le 4\ebrit_{k}(G)
        \text{ and }
        \mbrit_{k}(G)\le \ebrit_{k}(G).\] 
	
	Let $P=(X_1,X_2,\ldots,X_t)$ be a partition of $V(G)$
        such that $\abs{X_i}\le k$ for all $i$ 
        and the $\eta_{G}$-width of $P$ is $\ebrit_{k}(G)$.
        Then, the number of edges meeting two parts of $P$ is at most $2\ebrit_{k}(G)$. 
        Now, we take a partition $P'$ of $E(G)$ such that 
        for each $i\in \{1, 2, \ldots, t\}$, all the edges in $G[X_i]$ form one part of $P'$, 
        and individual edges meeting two parts of $X_1, \ldots, X_t$ form individual parts.
        Then $P'$ has $\kappa_G$-width at most $4\ebrit_{k}(G)$ and
        each part of $P'$ has at most $\binom{k}{2}$ edges. Thus we conclude that 
        $\vbrit_{\binom{k}{2}}(G)\le 4\ebrit_{k}(G)$.

        Note that for every vertex set $A$ of $G$, $\nu_G(A)\le \eta_G(A)$.
        Thus, $P$ has $\nu_G$-width at most $\ebrit_{k}(G)$, which implies that $\mbrit_{k}(G)\le \ebrit_{k}(G)$.   
        
        \medskip\noindent(2)
        The graph $(2\ell+1)K_{k}$ has edge $k$-brittleness $0$, while it has vertex $(\binom{k}{2}-1)$-brittleness at least $\ell+1$ by Lemma~\ref{lem:etabase1} and has matching $(k-1)$-brittleness at least $\ell+1$ by Lemma~\ref{lem:edgeforward1}.
\end{proof}
\subsection{Vertex $k$-scatteredness and matching $k$-scatteredness}
\begin{proposition}\label{prop:vertex-matching}
        \begin{enumerate}[(1)]
                \item    Every vertex $k$-scattered class of graphs is  matching $(k+1)$-scattered.
                \item    For every positive integer $k$, there exists a vertex $k$-scattered class of graphs that is not matching $k$-scattered.
                \item If a class of graphs is matching $k$-scattered for some integer $k$, then there exists an integer $k'$ such that it is vertex $k'$-scattered.
        \end{enumerate}
\end{proposition}
\begin{proof}
        \noindent (1)
        We claim that \[ \mbrit_{k+1}(G)\le 2\vbrit_{k}(G).\]
        Let  $P=(X_1,X_2,\ldots,X_t)$ be a partition of $E(G)$
        such that $\abs{X_i}\le k$ for all $i$
        and the $\kappa_{G}$-width of $P$ is $\vbrit_{k}(G)$.
        By the probabilistic argument, there are at most $2\vbrit_{k}(G)$ vertices meeting at least two parts of $P$.
        Let $S$ be the set of vertices incident with edges meeting at least two parts of $P$.
        Since no vertex of $G-S$ meets at least two parts of $P$, 
        each connected component $H$ of $G-S$ has at most $k$ edges and at most $k+1$ vertices. 
        Now, we take a partition $P'$ of $V(G)$ so that the vertex set of each connected component of $G-S$ forms a part, and vertices in $S$ form individual parts. 
        It is not hard to see that $P'$ has $\eta_G$-width at most $\abs{S}\le 2\vbrit_{k_2}(G)$.
        Thus, $\mbrit_{k+1}(G)\le 2\vbrit_{k}(G)$.
        
        \medskip\noindent(2)
        The graph $(2\ell+1)P_{k+1}$ has vertex $k$-brittleness $0$ while it has matching $k$-brittleness at least $\ell+1$ by Lemma~\ref{lem:edgeforward2}.

        \medskip\noindent(3)
        We claim that 
        \[\text{if }\mbrit_k(G)\le m,\text{ then }
        \vbrit_{\binom{4m+k}{2}}(G)\le 4m.\]
        Let $P=(X_1,X_2,\ldots,X_t)$ be a partition of $V(G)$
        such that $\abs{X_i}\le k$ for all~$i$
        and the $\nu_{G}$-width of $P$ is at most $m$.
        Let $k'=\binom{4m+k}{2}$.
        Let $H$ be the subgraph of $G$ consisting of edges meeting two parts of $P$.
        Let $M$ be a maximum matching of $H$.
        If $\abs{M}>2m$, then 
        by the probabilistic argument, there is a subset $\mathcal I$ of $\{1,2,\ldots,t\}$ such that at least half of the edges in $M$ joins a vertex in $X_i$ for some $i\in \mathcal I$ to a vertex in $X_j$ for some $j\notin \mathcal I$, contradicting the assumption that $\nu_G$-width of $P$ is at most $m$. 
 
        Thus, $\abs{M}\le 2m$. 
        Let $S$ be the set of ends of $M$. Then $\abs{S}\le 4m$
        and $S$ meets every edge of $H$.
        Then every component of $G-S$ is a subset of $X_i$ for some $i$
        and so has at most $k$ vertices.
        Now, we take a partition $P'$ of $E(G)$ so that 
        for each component $C$ of $G-S$, the set of edges incident with a vertex in $C$ forms a part,
        and the edges joining two vertices of $S$ form individual parts. 
        Then each part of $P'$ has at most $\binom{4m+k}{2}$ edges and 
        no vertex outside of $S$ meets more than one part of $P'$,
        meaning that $\kappa_G$-width of $P'$ is at most $4m$. 
        Thus, $\vbrit_{\binom{4m+k}{2}}(G)\le 4m$.
\end{proof}
\begin{proposition}\label{prop:matching-rank}
        \begin{enumerate}[(1)]
                \item Every  matching $k$-scattered class of graphs is rank $k$-scattered.
                \item For every integer $k>1$, there exists a matching $k$-scattered class of graphs that is not rank $(k-1)$-scattered.
        \end{enumerate}
\end{proposition}
\begin{proof}
        Observe that if a square $0$-$1$ matrix is non-singular,
        then the corresponding bipartite graph has a perfect matching.
        Thus, if a binary matrix $M$ has rank $r$,
        then its corresponding bipartite graph has a matching of size $r$.
        Thus, for all $S\subseteq V(G)$, $\cutrk_G(S)\le \nu_G(S)$.
        This implies that $\rkbrit_k(G)\le \mbrit_k(G)$.

        It is easy to see (2) from $\{nK_{k}:n\ge 1\}$ by Lemma~\ref{lem:edgeforward1}.
\end{proof}

\subsection{Tree-depth}                
A \emph{rooted forest} is a forest in which every connected component has  a specified node called a \emph{root}. The \emph{closure} of a rooted forest $T$ is the graph obtained from $T$ by adding an edge between every vertex and all its ancestors.
The \emph{height} of a rooted forest is the number of vertices
in  a longest path from a root to a leaf.
The \emph{tree-depth} of a graph  $G$, denoted by $\td(G)$, is
the minimum height of a rooted forest whose closure contains $G$ 
as a subgraph, see the book \cite[Chapter 6]{NO2012}.

Let us show that 
every matching $k$-scattered class of graphs has bounded tree-depth.

\begin{proposition}\label{prop:tdmbrit}
        Every matching $k$-scattered class of graphs has bounded tree-depth.
\end{proposition}

\begin{proof}

It is enough to prove that \[\td(G)\le 4\mbrit_k(G)+k.\]
Let $P=(X_1,X_2,\ldots,X_t)$ be a partition of $V(G)$
          such that $\abs{X_i}\le k$ for all $i$
          and the $\nu_{G}$-width of $P$ is $\mbrit_k(G)$.
          Let $M$ be a maximal matching of $G$ such that 
          every edge of $M$ is incident with two sets of $\{X_1, X_2, \ldots, X_t\}$.
          If $\abs{M}\ge 2\mbrit_k(G)+1$, then there exists a subset $\mathcal{I}$ of $\{1, 2, \ldots, t\}$ such that 
          at least $\mbrit_k(G)+1$ edges of $M$ are incident with 
          both $\bigcup_{i\in \mathcal{I}} X_i$ and $V(G)\setminus (\bigcup_{i\in \mathcal{I}} X_i)$, 
          which implies that the $\nu_{G}$-width of $P$ is at least  $\mbrit_k(G)+1$, a contradiction.
          Therefore, $\abs{M}\le 2\mbrit_k(G)$.
          
          Let $U$ be the set of all vertices incident with an edge of $M$.
          Then $\abs{U}\le 4\mbrit_k(G)$.
          By the choice of $M$, $G-U$ has no edges incident with two parts of $P$.
		So, $G-U$ has tree-depth at most $k$ and $G$ has tree-depth at most $4\mbrit_k(G)+k$.
\end{proof}

By Proposition~\ref{prop:tdmbrit}, 
every matching $k$-scattered class of graphs  has 
bounded path-width and bounded tree-width,
due to the inequality $\tw(G)\le \pw(G)\le \td(G)-1$ \cite{BGHK1995}, where $\tw$ denotes the tree-width and $\pw$ denotes the path-width.
\begin{proposition}
        There is a class of graphs of bounded tree-depth
        that is not rank $k$-scattered for any integer $k$.
\end{proposition}
\begin{proof}
        The graph $mK_{1,n}$ has tree-depth $2$ and yet its rank $k$-brittleness is at least $m/2$ when $n\ge k$ by Lemma~\ref{lem:edgeforward1}.
\end{proof}

\subsection{Shrub-depth and rank-depth}  
As a dense analogue of tree-depth, Ganian, Hlin\v{e}n\'{y}, Ne\v{s}et\v{r}il, Obdr\v{z}\'{a}lek, and
Ossona~de Mendez~\cite{GanianHNOO2019} proposed the notion of shrub-depth.
DeVos, Kwon, and Oum~\cite{DKO2019} introduced the notion of rank-depth of $G$ as the branch-depth of $ \cutrk_G$, and 
showed that a class of graphs has bounded rank-depth if and only if it has bounded shrub-depth.
So we will omit the definition of shrub-depth and review the definition of branch-depth instead.

A \emph{radius} of a tree is the minimum $r$ such that there is a node having distance at most $r$ from every node.
For a function $\lambda:2^E\to\mathbb{Z}^{\ge 0}$ on the subsets of a finite set $E$, 
a \emph{decomposition} of $\lambda$ is 
a pair $(T,\sigma)$ of a tree $T$ with at least one internal node
and a bijection $\sigma$ from $E$ to the set of leaves of $T$.
The \emph{radius} of a decomposition $(T,\sigma)$ 
is defined to be the radius of the tree $T$. 
For an internal node $v \in V(T)$, the components of the graph $T - v$ give rise to a partition $\mathcal P_v$ of $E$ by $\sigma$. 
The \emph{width} of $v$ is defined to be 
\[\max_{ \mathcal{P'} \subseteq \mathcal P_v } \lambda \left( \bigcup_{X \in \mathcal{P'} } X \right).\]
The \emph{width} of the decomposition $(T,\sigma)$ is the maximum width of an internal node of $T$.  We say that 
a decomposition $(T,\sigma)$ is a $(k,r)$-\emph{decomposition} of $\lambda$ if the width is at most $k$ and the radius is at most~$r$.
The \emph{branch-depth} of $\lambda$ is the minimum $k$ such that there exists a $(k,k)$-decomposition of $\lambda$.
If $\abs{E}<2$, then there exists no decomposition and we define $\lambda$ to have branch-depth $\lambda(\emptyset)$.

We denote by $\rd(G)$ the rank-depth of a graph, that is the branch-depth of $\cutrk_G$. We now prove that every rank $k$-scattered class of graphs has bounded rank-depth.

\begin{proposition}\label{prop:rdrbrit}
        Every rank $k$-scattered class of graphs has bounded rank-depth.
\end{proposition}

\begin{proof}
        We claim that \[\rd(G)\le  \max (k, \rkbrit_k(G), 2).\]
        Let $P=(X_1,X_2,\ldots,X_t)$ be a partition of $V(G)$
          such that $\abs{X_i}\le k$ for all $i$ 
          and the $\cutrk_{G}$-width of $P$ is $\rkbrit_k(G)$.
          We can obtain a $( \max (k, \rkbrit_k(G)) , 2)$-decomposition for $\cutrk_G$ as follows.
          Let $T$ be a tree obtained from $K_{1,t}$ with center $r$ and leaves $r_1, r_2, \ldots, r_t$
          by attaching $\abs{X_i}$ leaves to $r_i$ for each $i$.
          We map all vertices of $X_i$ to distinct leaves adjacent to $r_i$.
          Then the width of $r$ is $\rkbrit_k(G)$ and the width of $r_i$ is 
          at most $k$.
\end{proof}
\subsection{Linear rank-width}
Let us present the definition of linear rank-width \cite{Ganian2011,JKO2014,Oum2016}.
For a graph $G$, an ordering $(x_1, \ldots, x_n)$ of the vertex set $V(G)$ is called a \emph{linear layout} of $G$.  
	If $\abs{V(G)}\ge 2$, then the \emph{width} of a linear layout $(x_1,\ldots, x_n)$ of $G$ is defined
as
$\displaystyle\max_{1\le i\le n-1}\cutrk_G(\{x_1,\ldots,x_i\})$,
and if $\abs{V(G)}=1$, then the width is defined to be $0$.
	The \emph{linear rank-width} of $G$, denoted by $\lrw(G)$, is defined as the minimum width over all linear layouts of $G$. 
	For two orderings $(x_1, \ldots, x_n)$, $(y_1, \ldots, y_m)$, we write $(x_1, \ldots, x_n)\oplus (y_1, \ldots, y_m):=(x_1, \ldots, x_n, y_1, \ldots, y_m)$
        to denote the concatenation of two orderings.

We now aim to obtain an inequality between linear rank-width and rank $k$-brittleness.
Kwon, McCarty, Oum, and Wollan~\cite{KMSW2019} observed that $\lrw(G)\le \rd(G)^2$, and combining it with Proposition~\ref{prop:rdrbrit}, we can obtain a quadratic upper bound of linear rank-width in terms of rank $k$-brittleness.
Instead, we will obtain a linear upper bound directly.
For that, we use the submodularity of the matrix rank function.
	
	\begin{proposition}[See {\cite[Proposition 2.1.9]{Murota2000}}]\label{prop:submodularity}
	Let $M$ be a matrix over a field $\mathbb F$. Let $C$ be the set of column indexes of $M$, and 
	$R$ be the set of row indexes of $M$. Then for all $X_1, X_2\subseteq R$ and $Y_1, Y_2\subseteq C$, 
	\begin{multline*}
          \rank( M[X_1, Y_1])+\rank (M[X_2, Y_2])\ge\\
          \rank (M[X_1\cap X_2, Y_1\cup Y_2])+ \rank (M[X_1\cup X_2, Y_1\cap Y_2]).
        \end{multline*}
	\end{proposition}
	
	\begin{proposition}\label{prop:inequality}
          For every integer $k>0$,
          the linear rank-width of a graph $G$ is at most 
	\( \beta_k^\rho (G)+\left\lfloor k/2 \right\rfloor\).
	\end{proposition}
	\begin{proof}
	Let $x:=\beta_k^\rho (G)$.
        By the definition of rank $k$-brittleness, 
	there exists a partition $(X_1, X_2, \ldots, X_t)$ of $V(G)$ such that
       for each $i\in \{1, 2, \ldots, t\}$, $\abs{X_i}\le k$, and
      for every $I\subseteq \{1, 2, \ldots, t\}$, 	
	$\cutrk_G(\bigcup_{i\in I} X_i)\le x$.
	For each $i\in \{1, 2, \ldots, t\}$, let $L_i$ be any ordering of $X_i$.
	
	We claim that the ordering $L=L_1\oplus L_2\oplus \cdots \oplus L_t$ is a linear layout of $G$ having width at most $x+ \left\lfloor k/2 \right\rfloor$.
        It suffices to prove that for each $i\in \{1,2,\ldots,t\}$
        and a partition $(A,B)$ of $X_i$,
        $\rho_G(A\cup \bigcup_{j<i} X_j)\le x+\lfloor k/2\rfloor$.
        By symmetry, we may assume that $\abs{A}\le \lfloor k/2\rfloor$.
        Let $X=\bigcup_{j<i} X_j$ and $Y=V(G)\setminus X$.
        Let $M$ be the adjacency matrix of $G$.
	By Proposition~\ref{prop:submodularity}, 
	\begin{align*}
          \cutrk_G(A\cup X)&=\rank M[A\cup X, Y\setminus A]
                             + \rank M[\emptyset, Y] \\
                           &\le
                             \rank M[X, Y]+ \rank M[A,Y\setminus A] 
                           \le x+  \left\lfloor k/2 \right\rfloor.
	\end{align*}
        This proves the proposition.
	\end{proof}
	As the rank-width~\cite{Oum2004} of a graph is always less than or equal to its linear rank-width, we can deduce that the rank-width of a graph $G$ is at most 
	$\beta_k^\rho (G)+\left\lfloor k/2 \right\rfloor$.

 \begin{proposition}
        There is a class of graphs of modular-width $1$ that has unbounded linear rank-width. 
 \end{proposition}
 \begin{proof}
        Graphs of modular-width $1$ are precisely cographs~\cite{CLB1981}
        and cographs have unbounded linear rank-width, shown by Gurski and Wanke~\cite{GW2005a}. 
 \end{proof}
 
\section{An application}\label{sec:appl}
As an application of Theorem~\ref{thm:main4}, we prove that 
for fixed positive integers $m$ and $n$, 
$mK_{1,n}$-vertex-minor free graphs have bounded linear rank-width.
We will use the fact that every sufficiently large connected graph contains either a vertex of large degree or a long induced path.

\begin{proposition}[See Diestel~{\cite[Proposition 1.3.3]{Diestel2010}}]\label{prop:binary}
  For integers $k>3$ and $\ell>0$,  
  every connected graph on at least %
  $\frac{k-1}{k-3}(k-2)^{\ell-2}$
  vertices contains 
  a vertex of degree at least $k$ or an induced path on $\ell$ vertices.
\end{proposition}

Now we are ready to deduce Theorem~\ref{thm:main5} from Theorem~\ref{thm:main4} and Proposition~\ref{prop:inequality}.

\begin{thm:main5}
  For positive integers $m$ and $n$,  
  the class of graphs having no vertex-minor isomorphic to $mK_{1,n}$ has bounded linear rank-width.
\end{thm:main5}
\begin{proof}
  We may assume that $n\ge 3$.
Trivially $K_{1,n}$ is locally equivalent to $K_{n+1}$.
  By Lemma~\ref{lem:lengthonecase}, $P_{2n}$ is locally equivalent to $\S_n\tri\S_n$, and a vertex of degree $n$ in $\S_n\tri\S_n$ gives  a vertex-minor isomorphic to $K_{1,n}$.
  Therefore, by Proposition~\ref{prop:binary}, every connected graph
  on at least
  $ \frac{R(n;2)-1}{R(n;2)-3}(R(n;2)-2)^{2n-2}$
  vertices has a vertex-minor isomorphic to $K_{1,n}$.
  
  Let $k:=\lceil \frac{R(n;2)-1}{R(n;2)-3}(R(n;2)-2)^{2n-2}\rceil-1$.
  Let $\cC$ be the class of graphs having no $mK_{1,n}$ as a vertex-minor.
  Then for every connected graph $H$ on $k+1$ vertices,
  $mH\notin \cC$.
  Therefore by Theorem~\ref{thm:main4},
  $\cC$ is rank $k$-scattered.
  By Proposition~\ref{prop:inequality},
  $\cC$ has bounded linear rank-width.
\end{proof}

\paragraph{Acknowledgement.}
The authors would like to thank anonymous reviewers for their careful reviews and suggestions.

\end{document}